\documentclass[11pt,oneside]{amsart}
\usepackage{amsmath, amsfonts, amssymb}
\usepackage{graphicx, amsthm}
\usepackage{dsfont}
\usepackage{tikz}
\usepackage[normalem]{ulem}
\usepackage{enumitem}

\usepackage[colorlinks]{hyperref}

\providecommand{\U}[1]{\protect \rule{.1in}{.1in}}

\newtheorem{theorem}{Theorem}[section]
\newtheorem*{theorem*}{Theorem}

\newtheorem{lemma}[theorem]{Lemma}
\newtheorem{proposition}[theorem]{Proposition}
\theoremstyle{definition}

\newtheorem{remark}[theorem]{Remark}
\newtheorem{example}[theorem]{Example}

\newcommand{\lipnorm}[1]{\left\|#1\right\|_\textup{Lip}}  
\newcommand{\lipfree}[1]{\mathcal{F}({#1})}
\newcommand{\Lip}[1]{{\textup{Lip}(#1)}}
\newcommand{\restricted}[2]{#1\mathord{\upharpoonright_{#2}}}
\newcommand{\loclipnorm}[1]{\left\|#1\right\|_\textup{loc}}
\newcommand{\locLip}[1]{\textup{Lip}^{\textup{loc}}(#1)}
\newcommand{\LipO}[1]{{\textup{Lip}_{0}(#1)}}
\newcommand{\locLipO}[1]{\textup{Lip}_{0}^{\textup{loc}}(#1)}
\newcommand{\supp}{\textup{supp}}

\usepackage{soul,xcolor}

\author[G. Flores]{Gonzalo Flores}
\address[G. Flores]{Universidad de O'Higgins, Unidad de Acompañamiento Estudiantil, Av. Lib. Gral. Bernardo O'Higgins 611, Rancagua, Chile}
\email{gonzalo.flores@uoh.cl}

\begin{document}

\title[An isometric representation for $\lipfree{M}$ for length spaces embedded in $\mathbb{R}^{n}$]{An Isometric Representation for the Lipschitz-free Space of Length Spaces Embedded in Finite-dimensional Spaces}

\subjclass[2020]{Primary 46B04, 46B10; Secondary 46F10}

\keywords{Lipschitz-free spaces, Length space, Lipschitz function, Isometric representation}

\begin{abstract}
    For a domain $\Omega$ in a finite-dimensional space $E$, we consider the metric space $M=(\Omega,d)$ where $d$ is the intrinsic distance in $\Omega$. We obtain an isometric representation of the space $\LipO{M}$ as a subspace of $L^{\infty}(\Omega;E^{*})$ and we use this representation in order to obtain the corresponding isometric representation for the Lipschitz-free space $\lipfree{M}$ as a quotient of the space $L^{1}(\Omega;E)$. We compare our result with those existent in the literature for bounded domains with Lipschitz boundary, and for convex domains, which can be then deduced as corollaries of our result.
\end{abstract}

\maketitle

\section{Introduction}\label{sec1}

In recent years, Lipschitz-free spaces have been an active research topic in different contexts, such as computer science, optimization and Banach space geometry. The construction of these spaces can be found in these areas with different names for both the space and its norm. More precisely, Lipschitz-free spaces can be found as Arens-Eells spaces \cite[Chapter 3]{W18}, and are closely related to the Wasserstein-1 space \cite[Chapter 5]{S15}. In this context, for a finite diameter metric space $X$, the Arens-Eells space is defined starting from molecules defined over $X$, that is, finitely supported functions $f:X\to\mathbb{R}$ such that $\sum_{p\in X} f(p)=0$, while for compact metric spaces, the Wasserstein-1 space is defined in terms of probability measures over $X$ with finite first moment. These spaces are endowed with the Arens-Eells norm and Kantorovich-Rubinstein metric, respectively, which makes the first a predual of the space of Lipschitz functions defined over $X$ which vanish at a fixed point $x_{0}$, while the second one embeds isometrically into $\mathcal{F}(X)$ via the Kantorovich-Rubinstein duality. 

In general, each area has developed research pointing to results which are specific for the topic at hand. Nevertheless, some recent results point to the study of these spaces considering the different areas of research altogether and how some results from one of the areas can be applied to the others, as can be seen for example in \cite{OO19,OO20,OO22,OO24}

In the following, we will focus mainly on the study of Lipschitz-free spaces from the perspective of Banach space geometry and metric geometry. In this sense, we refer to the seminal paper Lipschitz-free Banach spaces by G. Godefroy and N. Kalton \cite{GK03}, which reestablished Lipschitz-free spaces as an active research topic. There, we can find the different known results for these spaces at that time, such as the existence of a unique extension for Lipschitz functions defined over two given metric spaces on the base space to a linear operator defined between the corresponding Lipschitz-free spaces, lifting properties, approximation properties, etc. The results in this direction have been complemented in \cite{AP20,CDW16,DKO19}.

In the present work, we consider the following definition of the Lipschitz-free space. For a pointed metric space $(M,\rho)$, that is, a metric space endowed with a distinguished point, usually denoted by $x_{0}$, which we refer to as the base point of $M$, define
\[ \LipO{M} := \{ f:M\to\mathbb{R} : f \text{ Lipschitz and } f(x_{0})=0 \}. \]
When endowed with
\[ \|f\|_{L} := \sup_{x,y\in M, x\neq y} \frac{f(y)-f(x)}{\rho(x,y)},\]
this space becomes a (dual) Banach space. The Lipschitz-free space is then obtained as its canonical predual, that is, the closed linear span in $\LipO{M}^{*}$ of the evaluation functionals. From this, it can be readily seen that $M$ isometrically embeds into $\lipfree{M}$ via the evaluation functionals.

From this last observation, it is expected that geometric properties of $\lipfree{M}$ might be affected by the geometry of the host space $M$. When $M$ is embedded in a finite dimensional space, isometric representations for $\lipfree{M}$ can be found in \cite[Theorem 2.4]{BCJ05} and \cite[Theorem 1.1]{CKK17}. More precisely, \cite[Theorem 2.4]{BCJ05} states that when $M=\overline{\Omega}$ for $\Omega\subset E$ a bounded domain with Lipschitz boundary endowed with its intrinsic distance, $\lipfree{M}$ is linearly isometric to the quotient of $L^{1}(\Omega;E)$ with respect to the subspace of functions with null divergence in the sense of distribution in $E$. On the other hand, \cite[Theorem 1.1]{CKK17} states that the same isometric representation is true when $M=\Omega\subset E$ is a convex domain endowed with the metric induced by the norm of $E$.

It is not difficult to see that none of these results implies the other. Nevertheless, the sets $\Omega$ considered in each case share some properties, such as being connected, having Lipschitz boundary and being endowed with the intrinsic distance (in the convex case, this distance coincides with that induced by the norm of $E$). Having this properties in mind, we propose a common ground to obtain a generalization for both \cite[Theorem 2.4]{BCJ05} and \cite[Theorem 1.1]{CKK17}, which goes as follows

\begin{theorem*}
    Let $\Omega\subset E$ be a domain and consider $M$ the set $\Omega$ endowed with its intrinsic distance and fix $x_{0}\in\Omega$ as the base point of $M$. Then, $\lipfree{M}$ is linearly isometric to $L^{1}(\Omega;E)/X$, where
    \[ X := \{ h\in L^{1}(\Omega;E) : \textup{div}(h) = 0 \text{ in } \mathcal{D}'(E) \}. \]
    Moreover, if $S$ is the preadjoint of the linear isometry $Tf:=\nabla f$, then it holds that $S[h] = \delta(x)$ if and only if $-\textup{div}(h) = \delta_{x}-\delta_{x_{0}}$ in $\mathcal{D}'(\Omega)$, where for $x\in\Omega$, $\delta_{x}$ is the Dirac distribution centered at $x$ and $[h]$ stands for the equivalence class of $h$ in $L^{1}(\Omega;E)/X$.
\end{theorem*}

Before continuing in this line, it is worth to mention for contrast some recent results found in the literature for the case where the metric space $M$ is purely $1$-unrectifiable, that is, $M$ does not contain any curve fragment (for the details of the definition, see e.g. \cite{B20}). In this case, the subspace $\textup{lip}_{0}(M)\subset\LipO{M}$ of locally flat functions captures the absence of curves. In this context, \cite[Theorem B]{AGPP22} states that $M$ is purely $1$-unrectifiable if and only if $\lipfree{M}\equiv (\textup{lip}_{0}(M))^{*}$. Moreover, \cite[Theorem C]{AGPP22} states that in $\lipfree{M}$, the Radon-Nikod\'ym, Krein-Milman and Schur properties are equivalent to each other, which are also equivalent to $\lipfree{M}$ containing no isomorphic copy of $L^{1}$, and also to the completion of $M$ being purely $1$-unrectifiable.

Notice that, in a way, purely $1$-unrectifiable spaces are the complete opposite to domains in $E$, in the sense that in a domain every two points can be joined by not only a Lipschitz path, but infinitely many. Isometries for Lipschitz-free spaces have been also studied under different hypotheses for the underlying metric space in \cite{DKP16,G10,NS07}.

This paper is outlined as follows. In Section \ref{section2} we establish the notation, some preliminaries on Lipschitz and locally Lipschitz functions, and finally we establish the tools that will be used in order to obtain the desired results. In Section \ref{section3}, we begin by constructing the framework to obtain a representation for the space $\LipO{M}$, and we show that this space is isometric to the subspace of $L^{\infty}(\Omega;E^{*})$ consisting on conservative vector fields in the appropriate sense. We compare this representation with the one obtained in \cite[Proposition 3.2]{CKK17}. Using this representation, we proceed to compute a predual for the subspace of $L^{\infty}(\Omega;E^{*})$ representing $\LipO{M}$ and we show that this predual is actually isometric to $\lipfree{M}$, obtaining the main result of the present work. Finally, in Section \ref{section4} we review the relation between our results and those found in \cite{BCJ05} and \cite{CKK17}.

\section{Framework and notation}\label{section2}

\subsection{Notation}\label{notation}

In this section, we specify the notation that will be used in the following pages

\begin{itemize}
    \item $E$ stands for a finite-dimensional space of dimension $n$. $\Omega\subset E$ is a non-empty connected open set (that is, a domain in $E$).
    \item $\|\cdot\|$ denotes a norm on $E$.
    \item $|\cdot|$ is used for the absolute value, the Euclidean norm of a vector or the Lebesgue measure of a set, depending on the context.
    \item $B(x,r)$ and $\overline{B}(x,r)$ denote the open and closed ball centered at $x$ and with radius $r$, respectively. The associated space and metric will often be clear from context.
    \item For $U\subset E$, $\mathcal{D}(U)$ denotes the spaces of test (smooth and compactly supported) functions in $U$, and $\mathcal{D}'(U)$ the corresponding space of distributions.
    \item $\partial_{i}$ and $\nabla$ denote the partial derivative with respect to the $i^{\textup{th}}$ coordinate and the gradient, respectively, in the sense of distributions.
    \item For $p\in[1,\infty]$, $L^{p}(\Omega;E)$ stands for the space of $E$ valued $p$-integrable functions when $p<\infty$, and $E$ valued essentially bounded measurable functions when $p=\infty$. In particular, $(L^{1}(\Omega;E))^{*}\equiv L^{\infty}(\Omega;E^{*})$ (since $E$ has the Radon-Nikod\'ym property).
\end{itemize}
Any additional notation will be clarified when needed.

\subsection{Preliminaries}

Let $M,N$ be metric spaces, endowed with distances $\rho$ and $\rho'$, respectively. A function $f:M\to N$ is called Lipschitz whenever there exists a number $L>0$ such that $\rho'(f(x),f(y)) \leq L \rho(x,y)$. Any such number is known as a Lipschitz constant for $f$. The Lipschitz number or Lipschitz norm of $f$, denoted by $\lipnorm{f}$, is the greatest lower bound for the Lipschitz constants, or equivalently, the least upper bound for the metric slopes of the function $f$
\[ \lipnorm{f} := \sup_{x\neq y} \frac{\rho'(f(x),f(y))}{\rho(x,y)}. \]
Notice that in this general context, and despite the name and notation, $\lipnorm{\cdot}$ is not actually a norm, neither a seminorm, since $N$ is only a metric space and as a consequence $\Lip{M,N}$ lacks a linear structure. Nevertheless, if $N$ is a vector space, the set of $N$-valued Lipschitz function become a vector space and $\lipnorm{\cdot}$ becomes a seminorm over that vector space. More precisely, $\lipnorm{f}=0$ if and only if $f$ is constant. This abuse of notation will be justified by the context where $\lipnorm{\cdot}$ is used. In the following lines we state some general properties where $N$ is just a metric space, which in the sequel will be applied in the case $N=\mathbb{R}$ endowed with its usual metric.

In the same fashion as for a Lipschitz function, we say that $f$ is locally Lipschitz at $x\in M$ if there exists an open set $A\subset M$ containing $x$ such that $\restricted{f}{A}$ is Lipschitz. In this case, the local Lipschitz number is defined as
\[ \Lip{f,x} := \inf\left\{ \lipnorm{\restricted{f}{A}} : A\subset M \text{ open and } x\in A \right\}. \]
If $f$ is locally Lipschitz at every $x\in M$, we simply say that $f$ is locally Lipschitz. For such a function, we further say that it is uniformly locally Lipschitz if the local Lipschitz numbers are bounded in $M$ and we define its uniform local Lipschitz number (or norm) as
\[ \loclipnorm{f} := \sup_{x\in M} \Lip{f,x}. \]
Just as in the case of Lipschitz functions, $\loclipnorm{\cdot}$ is a seminorm over the vector space of $N$-valued uniformly locally Lipschitz functions whenever $N$ is a vector space, which we denote by $\locLip{M,N}$. In general, $\Lip{M,N}$ is contained in $\locLip{M,N}$, with $\loclipnorm{f}\leq\lipnorm{f}$ for every $f\in\Lip{M,N}$, but the converse is not true without additional requirements, e.g. compactness of the domain.

\begin{theorem}\cite[\textbf{Theorem 2.1.6, p. 101}]{CMN19}\label{locLip}
    Let $(M,\rho),(N,\rho')$ be metric spaces, suppose that $(M,\rho)$ is compact, and let $f:M\to N$ be a function. Then, $f$ is Lipschitz whenever it is locally Lipschitz.
\end{theorem}

\begin{proof}
    Suppose that $f:M\to N$ is a locally Lipschitz function which is not Lipschitz. Then, for every $n\in\mathbb{N}$ there exists $x_{n},y_{n}\in M$ such that
    \[ \rho'(f(x_{n}),f(y_{n})) > n\rho(x_{n},y_{n}). \]
    Notice that for every $n\in\mathbb{N}$, $x_{n}\neq y_{n}$. Since $M$ is compact, we may assume without loss of generality that both $x_{n}$ and $y_{n}$ are convergent to $x$ and $y$, respectively. Since $f$ is continuous, we deduce from the inequality that both sequences converge to the same point $z\in M$. Then, thanks to $f$ being locally Lipschitz around $z$, there exists $L>0$ and $k\in\mathbb{N}$ such that for every $n\geq k$
    \[ n\rho(x_{n},y_{n}) < \rho'(f(x_{n}),f(y_{n})) \leq L\rho(x_{n},y_{n}), \]
    which implies $n<L$ for every $n\geq k$, a contradiction. Thus, $f$ is Lipschitz.
\end{proof}

\begin{remark}\label{unifloclip}
    Since every Lipschitz function is uniformly locally Lipschitz, Theorem \ref{locLip} implies that locally Lipschitz functions defined over $M$ are indeed uniformly locally Lipschitz whenever $M$ is compact. Then, the lemma implicitly states that $\Lip{M,N}=\locLip{M,N}$ whenever $M$ is compact.
\end{remark}

In the case $M=[0,1]$ with the usual metric, we can go a little further.

\begin{lemma}\label{loclip01}
    Let $(N,\rho')$ be a metric space and $g:[0,1]\to N$. Then, if $g$ is locally Lipschitz, $g$ is uniformly locally Lipschitz and $\lipnorm{g}=\loclipnorm{g}$.
\end{lemma}

\begin{proof}
    Suppose that $g:[0,1]\to N$ is a locally Lipschitz function. We already know thanks to Theorem \ref{locLip} that $g$ is Lipschitz, since $[0,1]$ is compact. Moreover, considering Remark \ref{unifloclip} we have that $g$ is uniformly locally Lipschitz. We claim that
    \[ \lipnorm{g} = \loclipnorm{g}. \]
    Clearly, $\loclipnorm{g}\leq\lipnorm{g}$, so let us prove the inequality $\loclipnorm{g}\geq\lipnorm{g}$.  Let $u,v\in[0,1]$ with $u<v$, $\varepsilon>0$ and consider for every $t\in[u,v]$ the radius $r_{t}>0$ given by the supremum of radii $0<\delta<v-u$ such that
    \[ \Lip{g,t} \geq \lipnorm{\restricted{g}{B(t,\delta)}} - \varepsilon. \]
    Let $t_{1}=u$ and while $t_{i}\neq v$, define $t_{i+1}=\min\{ v, t_{i}+r_{t_{i}}/2 \}$. If this process is infinite, we see that the sequence $t_{i}$ would be strictly increasing and bounded. Let $t\in[u,v]$ be its limit. Then, if $t-t_{i}<\eta< r_{t}$, we see that $r_{t_{i}}\geq r_{t}-\eta$. Indeed, notice that if $s\in B(t_{i},r_{t}-\eta)$, then
    \[ |s-t|\leq |s-t_{i}|+|t_{i}-t| < r_{t}, \]
    that is $B(t_{i},r_{t}-\eta)\subset B(t,r_{t})$, which leads us to
    \[ \lipnorm{\restricted{g}{B(t_{i},r_{t}-\eta)}} \leq \lipnorm{\restricted{g}{B(t,r_{t})}} \leq \Lip{g,t} + \varepsilon, \]
    from which we deduce that $r_{t_{i}}\geq r_{t}-\eta$.
    
    Choosing $\eta=r_{t}/2$ we see that 
    \[ \liminf_{i\to\infty}{r_{t_{i}}}\geq \frac{r_{t}}{2}>0.\]
    But
    \[ t-u = \sum_{i=1}^{\infty} t_{i+1}-t_{i} = \frac{1}{2}\sum_{i=1}^{\infty} r_{t_{i}}. \]
    Since this sum is not convergent, we deduce that the process for obtaining the sequence $t_{i}$ eventually stops. We have then $u = t_{1}<\ldots<t_{k}=v$ such that $t_{i+1}\in B(t_{i},r_{t_{i}})$. We see that 
    \[ \frac{\rho'(g(u),g(v))}{v-u} \leq \sum_{i=1}^{k-1}\frac{\rho'(g(t_{i}),g(t_{i+1}))}{v-u} = \sum_{i=1}^{k-1} \frac{t_{i+1}-t_{i}}{v-u} \frac{\rho'(g(t_{i}),g(t_{i+1}))}{t_{i+1}-t_{i}} \]
    \[ \leq \sum_{i=1}^{k-1} \lambda_{i} \lipnorm{\restricted{g}{B(t_{i},r_{t_{i}})}} \leq \sum_{i=1}^{k-1} \lambda_{i} (\Lip{g,t_{i}}+\varepsilon), \]
    where $\lambda_{i} = (t_{i+1}-t_{i})/(v-u)>0$ and their sum is equal to 1. Since $g$ is uniformly locally Lipschitz, we deduce that
    \[ \frac{\rho'(g(u),g(v))}{v-u} \leq \varepsilon + \max_{i=1,\ldots,k} \Lip{g,t_{i}} \leq \varepsilon + \loclipnorm{g}, \]
    from which we get that $\lipnorm{g}\leq \loclipnorm{g}$, since the last inequality holds for every $u,v\in[0,1]$ and $\varepsilon>0$.
\end{proof}

Making use of these lemmas in an specific framework, we can find a broader class of metric spaces for which $\Lip{M,N}=\locLip{M,N}$. For a rectifiably-connected metric space $(M,\rho)$, the intrinsic (length) metric is defined as
\[ d_{\rho}(x,y) = \inf\{ \ell(\gamma) : \gamma\in \mathcal{C}([0,1];M), \gamma(0)=x, \gamma(1)=y \} \]
where
\[ \ell(\gamma) = \sup\left\{ \sum_{i=1}^{n} \rho(\gamma(t_{i-1}),\gamma(t_{i})) : 0=t_{0}<\ldots <t_{n} = 1 , n\in\mathbb{N} \right\} \]
It is well known that for every rectifiable path $\gamma:[0,1]\to M$ there exists an orientation-preserving reparametrization $\theta:[0,1]\to M$ such that $\theta$ has constant speed, that is, for every $u<v\in[0,1]$, the length of $\theta([u,v])$, denoted by $\ell_{uv}(\theta)$, satisfies $\ell_{uv}(\theta) = \ell(\theta)|v-u|$, which is given by the scaling of the length of arc parametrization of $\gamma$. Considering this, we will use directly Lipschitz paths whenever possible, that is, whenever we are dealing with parametrization-invariant properties. Noticing that length of arc parametrizations are also Lipschitz, we will also assume that parametrizations are Lipschitz when convenient. For more details on rectifiably-connected metric spaces, length spaces and the related definitions, we refer to \cite[Section 2.5]{IBB01}.

\begin{proposition}\label{unifloclip2}
    Let $(M,\rho)$ and $(N,\rho')$ be metric spaces, with $M$ rectifiably-connected. Then, if $\rho$ is equivalent to $d_{\rho}$, uniformly locally Lipschitz functions are Lipschitz. Moreover, for every $f\in\Lip{M,N}=\locLip{M,N}$, it holds $\loclipnorm{f}\leq\lipnorm{f}\leq C\loclipnorm{f}$, where $C$ is the equivalence constant for $\rho$ and $d_{\rho}$.
\end{proposition}

\begin{proof}
    It suffices to prove the nontrivial inclusion $\locLip{M,N}\subset\Lip{M,N}$. Let $x,y\in M$ and $\gamma\in \mathcal{C}([0,1];M)$ a rectifiable path going from $x$ to $y$. As usual, we assume without loss of generality that $\gamma$ has constant speed, that is, $\ell_{st}(\gamma)=\ell(\gamma)|t-s|$ for every $s,t\in[0,1]$. We claim that $f\circ \gamma$ belongs to $\locLip{[0,1],N}$. Indeed, consider $t\in[0,1]$. Since $\gamma(t)\in M$ and $f$ is locally Lipschitz, for every $\varepsilon>0$ there exists an open neighborhood $A\subset M$ of $\gamma(t)$ such that for every $x,y\in A$
    \[ \rho'(f(x),f(y))\leq (\Lip{f,\gamma(t)}+\varepsilon)\rho(x,y). \]
    Then, for $u,v\in \gamma^{-1}(A)$
    \[ \rho'(f\circ\gamma(u),f\circ\gamma(v))\leq (\Lip{f,\gamma(t)}+\varepsilon)\rho(\gamma(u),\gamma(v)) \]
    \[ \leq(\Lip{f,\gamma(t)}+\varepsilon)\ell_{uv}(\gamma) = (\Lip{f,\gamma(t)}+\varepsilon)\ell(\gamma)|v-u|. \]
    From this, $\lipnorm{\restricted{f\circ\gamma}{\gamma^{-1}(A)}}\leq (\Lip{f,\gamma(t)}+\varepsilon)\ell(\gamma)$, which consequently leads to $\Lip{f\circ\gamma,t}\leq \Lip{f,\gamma(t)}\ell(\gamma)$. By virtue of Theorem \ref{locLip}, $f\circ\gamma$ is Lipschitz in $[0,1]$, that is
    \[ \rho'(f\circ\gamma(u),f\circ\gamma(v)) \leq \lipnorm{f\circ\gamma}|v-u|. \]
    In particular, for $u=0$ and $v=1$ we get that
    \[ \rho'(f(x),f(y)) \leq \lipnorm{f\circ\gamma} \]
    Thanks to Lemma \ref{loclip01} we have that $\lipnorm{f\circ\gamma}=\loclipnorm{f\circ\gamma}$. Then
    \[ \rho'(f(x),f(y)) \leq \loclipnorm{f\circ\gamma} = \sup_{t\in[0,1]} \Lip{f\circ\gamma,t} \leq \ell(\gamma)\sup_{t\in[0,1]} \Lip{f,\gamma(t)}. \]
    But since $f$ is uniformly locally Lipschitz, we see that
    \[ \sup_{t\in[0,1]} \Lip{f,\gamma(t)} \leq \sup_{x\in M} \Lip{f,x} = \loclipnorm{f}. \]
    Finally, since $\gamma$ was an arbitrary constant speed path going from $x$ to $y$, we get that for $C>0$ such that $d_{\rho}(x,y)\leq C \rho(x,y)$ for every $x,y\in M$
    \[ \rho'(f(x),f(y)) \leq \loclipnorm{f}d_{\rho}(x,y)\leq C\loclipnorm{f}\rho(x,y). \]
    Hence, $f$ is Lipschitz and $\loclipnorm{f}\leq\lipnorm{f}\leq C\loclipnorm{f}$.
\end{proof}

\subsection{Framework}\label{framework}

In the present article, some classical results on vector calculus and differential geometry will be necessary, which we recall in the following propositions.

\begin{proposition}\label{prop:cons}
    Let $\Omega\subset E$ be a non-empty open set and $g:\Omega\to E$ be a continuous vector field. Then, the following are equivalent
    \begin{enumerate}
        \item $g$ is conservative (i.e. $g$ has a $\mathcal{C}^{1}(\Omega)$ potential) 
        \item For every pair of rectifiable paths $\gamma_{1},\gamma_{2}:[0,1]\to\Omega$ with common endpoints
        \[ \int_{\gamma_{1}} g\cdot d\vec{r} = \int_{\gamma_{2}} g\cdot d\vec{r} \]
        \item For every rectifiable loop $\gamma:[0,1]\to\Omega$ (i.e. such that $\gamma(0)=\gamma(1)$)
        \[ \oint_{\gamma} g\cdot d\vec{r} = 0 \]
        \item For every pair of piecewise linear paths $\gamma_{1},\gamma_{2}:[0,1]\to\Omega$ with common endpoints
        \[ \int_{\gamma_{1}} g\cdot d\vec{r} = \int_{\gamma_{2}} g\cdot d\vec{r} \]
        \item For every piecewise linear loop $\gamma:[0,1]\to\Omega$
        \[ \oint_{\gamma} g\cdot d\vec{r} = 0 \]
    \end{enumerate}
\end{proposition}

The chain of implications of $(1)\Rightarrow(2)\Rightarrow(3)\Rightarrow(1)$ can be widely found in the literature, involving simple techniques of vector calculus (see e.g. \cite[Chapter 9]{WT03}). The implication $(2)\Rightarrow(4)$ is trivial, while $(4)\Rightarrow(5)$ and $(5)\Rightarrow(1)$ are obtained in the same way as $(2)\Rightarrow(3)$ and $(3)\Rightarrow(1)$, respectively.

From differential geometry, in order to achieve our goal we need a specific result which is formulated in terms of $1$-forms, but can be readily translated to vector fields. More precisely, a particular version of Poincaré's Lemma generalizes a well known result for $\mathcal{C}^{1}$ conservative vector fields in $\mathbb{R}^{3}$, which states that if $\Omega\subset\mathbb{R}^{3}$ is a simply connected domain, then a $\mathcal{C}^{1}$ vector field is conservative if and only if it is irrotational (it has null curl). We recall the aforementioned version of Poincaré's Lemma in terms of vector fields, which will be useful in the sequel. For the details, we refer to \cite{dC94}.

\begin{lemma}[\textbf{Poincaré for $1$-forms}]\label{poincare}
    Let $\Omega\subset E$ be a simply connected domain, and $g:\Omega\to E$ a $\mathcal{C}^{1}$ vector field. Then
    \[ g \text{ is conservative} \iff \forall 1\leq i,j \leq n, \quad \partial_{i}g_{j} = \partial_{j}g_{i} \]
\end{lemma}

It is well known that whenever $g:\Omega\to E$ is a $\mathcal{C}^{1}$ conservative vector field, its Jacobian matrix is symmetric, that is
\[ \partial_{i}g_{j}=\partial_{j}g_{i} \text{ in } \Omega, \]
since it coincides with the Hessian matrix of any potential for $g$. The converse is false in general, which is clear from the following classical example

\begin{example}
    Consider the vector field $g:\mathbb{R}^{2}\setminus\{0\}\to\mathbb{R}^{2}$ given by
    \[ g(x,y) = \left( \frac{-y}{x^{2}+y^{2}} , \frac{x}{x^{2}+y^{2}} \right), \]
    which is clearly $C^{1}$ in its domain and satisfies
    \[ \partial_{1}g_{2}(x,y) = \frac{y^{2}-x^{2}}{(x^{2}+y^{2})^{2}} = \partial_{2}g_{1}(x,y) \]
    But its integral over the path $\gamma(t)=(\cos(2\pi t),\sin(2\pi t))$ is given by
    \[ \int g\cdot d\gamma = 2\pi \int_{0}^{1} \sin^{2}(2\pi t)+\cos^{2}(2\pi t) dt = 2\pi \neq 0, \]
    which shows that this field is not conservative. Notice that all of this still holds if we replace the domain by $\mathbb{R}^{2}\setminus \overline{B}(0,1/2)$ and in that case $g$ is bounded.
\end{example}

Next, we state the following result on measure theory which will come in handy for some of the proofs below.

\begin{lemma}\cite[\textbf{Lemma 8.3, p. 161}]{L93}\label{restae}
    Let $(X,\Sigma_{X},\mu),(Y,\Sigma_{Y},\nu)$ be measure spaces and consider $Z$ be a set of $(\mu\otimes\nu)$-measure $0$ in $X\times Y$. Then, for almost all $x\in X$ we have $\nu(Z_{x})=0$, where
    \[ Z_{x} = \{ y\in Y : (x,y)\in Z \}. \]
\end{lemma}


\section{Lipschitz-free space for domains}\label{section3}

\subsection{Identification of Lipschitz functions over a domain}\label{Lip0}

In the following, we will show that (in the appropriate sense) conservative essentially bounded vector fields have Lipschitz potential. Recall that a continuous vector field $g$ is said to be conservative if it has a potential, that is, there exists a $\mathcal{C}^{1}$ scalar field $f$ such that $\nabla f = g$.

Our goal is to weaken this condition in order to treat with non-continuous vector fields. Since we want to deal with Lipschitz functions, which will play the role of scalar fields, our definition will necessarily impose vector fields to be the gradients of Lipschitz functions. To this end, first recall the following well known result for Lipschitz functions.

\begin{theorem}\cite[\textbf{Section 5.8 Theorem 6 (Rademacher)}]{E10}
    Consider $f:\Omega\subset\mathbb{R}^{n}\to\mathbb{R}^{m}$ a locally Lipschitz function. Then, the partial derivatives $\partial_{i}f_{j}(x)$ exists almost everywhere in $\Omega$ with respect to the Lebesgue measure and they are Lebesgue-measurable functions. Moreover, they coincide with the partial derivatives of $f$ in the sense of distributions.
\end{theorem}

From now on, whenever $\Omega\subset E$ is a domain, we will understand that $\Omega$ is endowed with the metric $\rho$ induced by the norm, and we will denote by $M$ the metric space given by $\Omega$ endowed with its intrinsic metric $d$.

Recall that since $\Omega$ is open, $\rho$ and $d$ coincide locally in $\Omega$, thanks to the fact that $E$ is locally convex. This shows that every Lipschitz function over $M$ is differentiable a.e., since these functions are locally Lipschitz with respect to $\rho$. Using this fact, we consider from now on the gradient operator on $\LipO{M}$, that is, $T:\LipO{M}\to L^{\infty}(\Omega;E^{*})$ given by $Tf:=\nabla f$.

\begin{proposition}\label{Tbounded}
    The operator $T:\LipO{M}\to L^{\infty}(\Omega;E^{*})$ given by the gradient of $f$, $Tf:=\nabla f$, is well defined and is an injective bounded linear operator.
\end{proposition}

\begin{proof}
    Thanks to Rademacher theorem and the fact that the spaces $\LipO{M}$ and $\locLipO{M}$ functions coincide thanks to Proposition \ref{unifloclip2} ($M$ is endowed with its intrinsic distance), we deduce that $\nabla f$ exists a.e. in $\Omega$. Moreover, if $x\in\Omega$ is a differentiability point of $f$ and $h\in E$, then
    \[ |\langle \nabla f(x),h \rangle| = \left| \lim_{t\to 0^{+}} \frac{f(x+th)-f(x)}{t} \right| \]
    \[ = \lim_{t\to 0^{+}} \frac{|f(x+th)-f(x)|}{t} \leq \limsup_{t\to 0^{+}} \frac{\lipnorm{f}d(x,x+th)}{t}. \]
    But since $\Omega$ is open, $d(x,x+th)=\|x+th-x\|=t\|h\|$ for every $t>0$ sufficiently close to $0$, which finally leads to $|\langle \nabla f(x),h \rangle| \leq \lipnorm{f}\|h\|$. From this, $\|\nabla f(x)\|_{*}\leq\lipnorm{f}$ a.e. in $\Omega$, which implies that $\nabla f \in L^{\infty}(\Omega;E^{*})$, with $\|Tf\|_{\infty}\leq \lipnorm{f}$. Since the differential operator $\nabla$ is linear, we deduce that $T$ is well defined and is a bounded linear operator.

    For the injectivity of $T$, suppose that $Tf = 0$, that is, $\nabla f = 0$ a.e. in $\Omega$. Since $\Omega$ is connected, $f$ must be constant. But the only constant function in $\LipO{M}$ is $f=0$. Hence, $T$ is injective.
\end{proof}

Proposition \ref{Tbounded} shows that $T$ is a bijective bounded linear operator between $\LipO{M}$ and its image
\[ \textup{Im}(T) := \{ g\in L^{\infty}(\Omega;E^{*}) : (\exists f\in\LipO{M}) \quad g = \nabla f \}. \]

Our goal in the following is to establish a characterization of $\textup{Im}(T)$ and to show that the inverse of $T$ can be described in the expected natural way, that is, in terms of the appropriate integrals in order to recover $f$ from its gradient. Using the aforementioned characterization for $\textup{Im}(T)$ we will show that $T$ actually defines an isometry between $\LipO{M}$ and its image. In the following lines we give the details for recovering a smooth Lipschitz function (up to a constant) from its gradient.

\begin{example}
    Suppose that $f:\Omega\to \mathbb{R}$ is a $C^{\infty}(\Omega)$ function such that $\nabla f$ is bounded and fix $x_{0}\in\Omega$. If $\gamma:[0,1]\to\Omega$ is any Lipschitz path going from $x_{0}$ to some $x\in\Omega$, noticing that $f\circ\gamma:[0,1]\to\mathbb{R}$ is continuous and differentiable almost everywhere, we see that
    \[ f(\gamma(1))-f(\gamma(0)) = \int_{0}^{1} (f\circ\gamma)'(t) dt = \int_{0}^{1} \nabla f(\gamma(t)) \cdot \gamma'(t) dt, \]
    that is,
    \[ f(x) = f(x_{0}) + \int_{0}^{1} \nabla f(\gamma(t))\cdot\gamma'(t) dt. \]
    This shows that the obtained value for $f(x)$ is independent of $\gamma$, since the field $\nabla f$ is conservative by definition. Moreover, and considering that
    \[ f(y)-f(x) = \int_{0}^{1} \nabla f(\gamma_{xy}(t))\cdot \gamma'_{xy}(t) dt, \]
    where $\gamma_{xy}:[0,1]\to\Omega$ is any Lipschitz path going from $x$ to $y$, we deduce that
    \[ |f(y)-f(x)| \leq \int_{0}^{1} \|\nabla f(\gamma_{xy}(t))\|_{*} \|\gamma'_{xy}(t)\| dt \]
    \[ \leq \|\nabla f\|_{\infty} \int_{0}^{1}\|\gamma'_{xy}(t)\| dt = \|\nabla f\|_{\infty}\ell(\gamma_{xy}). \]
    Since this holds true for any Lipschitz path going from $x$ to $y$, we deduce that
    \[ |f(y)-f(x)| \leq \|\nabla f\|_{\infty} d(x,y), \]
    which shows that $\lipnorm{f}\leq\|Tf\|_{\infty}$ whenever $f$ is a $\mathcal{C}^{\infty}(\Omega)$ function with bounded gradient.
\end{example}

With this procedure we have stated two facts. First, when restricted to $\mathcal{C}^{\infty}(\Omega)$ functions vanishing at the base point $x_{0}$ of $M$ with bounded gradient, $T$ is in fact an isometry onto its image. Second, the inverse of this restriction of $T$ is defined in terms of integrals over curves in $\Omega$. These integrals are independent of the chosen path thanks to the fact that we are starting from the gradient of a $\mathcal{C}^{\infty}(\Omega)$ function, that is, we are dealing with a conservative vector field. 

Considering this idea, we claim that $T^{-1}$ is given by a similar formula, which takes into account the non-smoothness of the vector fields making use of the $\mathcal{C}^{\infty}$ case.

From now on, for every $k\in\mathbb{N}$ consider the set
\[ \Omega_{k} = \left\{ x\in\Omega : \textup{dist}(x,\Omega^{c})>\frac{1}{k} \right\} \]
and a positive mollifier $u:\mathbb{R}^{n}\to\mathbb{R}$, that is, a positive $\mathcal{C}^{\infty}$ function supported in the closed unit ball of $\mathbb{R}^{n}$ and such that $\int u(x)dx =1$. In this context, for every $k\geq 1$ the functions $u_{k}:\mathbb{R}^{n}\to\mathbb{R}$ are defined as $u_{k}(x):=k^{n}u(kx)$. Recall that a continuous vector field $H:U\subset\mathbb{R}^{n}\to\mathbb{R}^{n}$ defined over an open set is conservative in its domain if there exists a function $h:U\to\mathbb{R}$ such that $H=\nabla h$.

\begin{theorem}\label{main1}
    For the operator $T:\LipO{M}\to L^{\infty}(\Omega;E^{*})$ the following holds
    \begin{enumerate}[label=\roman*)]
        \item\label{main1_1} The range of $T$ is
        \[ Y := \{ g\in L^{\infty}(\Omega;E^{*}) : (\forall k\in\mathbb{N}) \quad g*u_{k}:\Omega_{k}\to\mathbb{R}^{n} \text{ is conservative} \}. \]
        \item\label{main1_2} The inverse operator $T^{-1}:Y\to\LipO{M}$ is given by
        \[ T^{-1}g(x) = \lim_{k\to\infty} \int_{0}^{1} (g*u_{k})(\gamma(t)) \cdot \gamma'(t) dt, \]
        where for each $x\in M$, $\gamma:[0,1]\to\Omega$ is any Lipschitz path going from the base point $x_{0}\in M$ to $x$. The choice of $\gamma$ is irrelevant thanks to the definition of $Y$.
    \end{enumerate}
    In particular, $T:\LipO{M}\to Y$ is an isometric isomorphism.
\end{theorem}

\begin{proof}

    From now on, we will make use of mollifiers, weak derivatives and distributions, and their properties, for which we refer to \cite{E10,FJ99}.
    We begin by proving $\it{\ref{main1_2}.}$. Notice that if $g=\nabla f$ for some $f\in\LipO{M}$, then for $k\in\mathbb{N}$ we see that $g*u_{k}\in\mathcal{C}^{\infty}(\Omega_{k};E^{*})$. Notice that choosing $\Omega_{k}$ as the domain of $g*u_{k}$ allows us to compute the value of
    \[ (g*u_{k})(x) = \int_{\overline{B}(0,1/k)} u_{k}(x)g(x-y) dy \]
    without needing to extend $g$ outside of $\Omega$. Moreover, for $x\in\Omega_{k}$
    \[ (g*u_{k})_{i}(x) = (g_{i}*u_{k})(x) = \langle \partial_{i} f, u_{k}(x-\cdot) \rangle = \langle f,\partial_{i}u_{k}(x-\cdot) \rangle = \partial_{i}(f*u_{k})(x), \]
    that is, $g*u_{k} = \nabla (f*u_{k})$, so by definition $g*u_{k}$ is conservative (which in turn proves one of the inclusions of $\it{\ref{main1_1}.}$, more precisely that $\textup{Im}(T)\subset Y$). Since, by virtue of Proposition \ref{Tbounded}, $T$ is a bounded linear bijection onto its image, in order to prove $\it{\ref{main1_2}.}$, it suffices to show that the operator given in $\it{\ref{main1_2}.}$ is a left inverse for $T$. 
    
    Let $g=\nabla f$ for some $f\in\LipO{M}$. Since $g*u_{k}=\nabla(f*u_{k})$, we see that for any Lipschitz path $\gamma$ going from the base point $x_{0}\in M$ to $x$
    \[ \int_{0}^{1} (g*u_{k})(\gamma(t)) \cdot \gamma'(t) dt = (f*u_{k})(x) - (f*u_{k})(x_{0}). \]
    Since $f$ is continuous, $f*u_{k}$ converges uniformly to $f$ over compact subsets of $\Omega$, and in particular, it converges pointwise to $f$. Then, recalling that $f(x_{0})=0$
    \[ \lim_{k\to \infty} \int_{0}^{1} (g*u_{k})(\gamma(t)) \cdot \gamma'(t) dt = f(x). \]
    We deduce from this the desired formula for $T^{-1}$. 
           
    For the remainder inclusion of $\it{\ref{main1_1}.}$, that is $Y\subset\textup{Im}(T)$, we will prove that provided $g*u_{k}$ is conservative for every $k\in\mathbb{N}$, there exists $f\in\LipO{M}$ such that for every $i\in\{ 1,\ldots,n \}$, $\partial_{i} f = g_{i}$ a.e. on $\Omega$. From this we can deduce directly that $\nabla f = g$ a.e. on $\Omega$. Let $x_{0}\in\Omega$ be the base point of $M$ and consider for every $k\in\mathbb{N}$ the set $\Omega_{k}^{x_{0}}$, defined as the connected component of $\Omega_{k}$ which contains $x_{0}$. Notice that both $\Omega_{k}$ and $\Omega_{k}^{x_{0}}$ are open for every $k$, and $\Omega_{k}^{x_{0}}$ is non-empty whenever $x_{0}\in\Omega_{k}$, that is, $1/k<\textup{dist}(x_{0},\Omega^{c})$. Moreover, $\Omega_{k}^{x_{0}}\subset\Omega_{k+1}^{x_{0}}$ for every $k\in\mathbb{N}$ and
    \[ \Omega = \bigcup_{k\in\mathbb{N}} \Omega_{k}^{x_{0}}, \]
    since for every $x\in\Omega$ there exists a Lipschitz path $\gamma:[0,1]\to\Omega$ going from $x_{0}$ to $x$, whose image satisfies $\gamma([0,1])\subset\Omega_{k}^{x_{0}}$ for every $k$ big enough. We endow each $\Omega_{k}^{x_{0}}$ with its intrinsic distance $d_{k}$, which is well defined as long as $\Omega_{k}^{x_{0}}\neq\emptyset$, and we denote by $M_{k}$ the resulting metric spaces. We see that for every $x,y\in\Omega_{k}^{x_{0}}$, the sequence $(d_{i}(x,y))_{i\geq k}$ is decreasing. Indeed, since $\Omega_{k}^{x_{0}}\subset\Omega_{i}^{x_{0}}$ for every $i\geq k$, we have that $d_{i}(x,y)$ is well defined. Moreover, since every Lipschitz path between $x$ and $y$ in $\Omega_{i}^{x_{0}}$ is also a Lipschitz path in $\Omega_{i+1}^{x_{0}}$, we deduce that $d_{i+1}(x,y)\leq d_{i}(x,y)$. By a similar reasoning, we deduce that for every $x,y\in\Omega$ and every $k\in\mathbb{N}$ big enough (it suffices that $x,y\in\Omega_{k}^{x_{0}}$), $d(x,y)\leq d_{k}(x,y)$. We claim that actually
    \[ d(x,y) = \inf_{k\in\mathbb{N}} d_{k}(x,y). \]
    To see this, let $x,y\in\Omega$ and $\varepsilon>0$. Then, there exists a Lipschitz path $\gamma:[0,1]\to\Omega$ joining $x$ and $y$ such that $\ell(\gamma)\leq d(x,y)+\varepsilon$. But then, there exists $i\in\mathbb{N}$ such that $\gamma([0,1])\subset\Omega_{k}^{x_{0}}$, for every $k\geq i$. Then we see that whenever $k\geq i$
    \[ d(x,y)\leq d_{k}(x,y)\leq \ell(\gamma) \leq d(x,y)+\varepsilon, \]
    which leads to $\displaystyle d(x,y)= \lim_{k\to\infty} d_{k}(x,y) = \inf_{k\in\mathbb{N}}d_{k}(x,y)$. Having stated this framework in $\Omega$, we proceed to construct the desired potential for $g$. For $k\geq m$ define the functions $f_{k}^{m}:\Omega_{m}^{x_{0}}\to\mathbb{R}$ as
    \[ f_{k}^{m}(x) = \int_{0}^{1} (g*u_{k})(\gamma(t))\cdot\gamma'(t) dt, \]
    where $\gamma:[0,1]\to\Omega_{m}^{x_{0}}$ is any Lipschitz path going from $x_{0}$ to $x$. Since we are assuming that $g*u_{k}$ is conservative, the choice of this path is irrelevant in the definition of $f_{k}^{m}$. For every $k\geq m$, $f_{k}^{m}\in\LipO{M_{m}}$, since for every Lipschitz path $\gamma_{xy}:[0,1]\to\Omega_{k}^{x_{0}}$
    \[ |f_{k}^{m}(y)-f_{k}^{m}(x)| = \left| \int_{0}^{1} (g*u_{k})(\gamma_{xy}(t))\cdot\gamma'_{xy}(t) dt \right| \]
    \[ \leq \int_{0}^{1} \|(g*u_{k})(\gamma_{xy}(t))\|_{*}\|\gamma'_{xy}(t)\| dt \]
    \[ \leq \|g*u_{k}\|_{\infty}\int_{0}^{1} \|\gamma'_{xy}(t)\| dt \leq \|g\|_{\infty}\ell(\gamma_{xy}). \]
    From this inequality and the fact that $\gamma_{xy}:[0,1]\to\Omega_{k}^{x_{0}}$ was arbitrary, we deduce that
    \[ |f_{k}^{m}(y)-f_{k}^{m}(x)| \leq \|g\|_{\infty}d_{k}(x,y)\leq \|g\|_\infty d_m(x,y), \]
    proving that $f_{k}^{m}\in\LipO{M_{m}}$ for every $k\geq m$ and also that for fixed $m$, the sequence $(f_{k}^{m})_{k\geq m}$ is bounded by $\|g\|_{\infty}$ in $\LipO{M_{m}}$. We proceed to construct then a sequence of functions $(f^{m})_{m\in\mathbb{N}}$, with $f^{m}\in\LipO{M_{m}}$ by induction. Suppose without loss of generality that $\Omega_{1}^{x_{0}}\neq\emptyset$ and define $f_{1}$ as the pointwise limit of a subsequence of $(f_{k}^{1})_{k\geq 1}$. This is possible thanks to Banach-Alaoglu theorem, since $\lipnorm{f_{k}^{1}}$ is bounded, and the fact that the weak-$*$ topology of $\LipO{M_{1}}$ coincides with the topology of pointwise convergence on bounded sets. More precisely, for some increasing function $\alpha_{1}:\mathbb{N}\to\mathbb{N}$ we have that
    \[ (\forall x\in\Omega_{1}^{x_{0}}) \quad f^{1}(x) = \lim_{k\to\infty} f^{1}_{\alpha_{1}(k)}(x). \]
    Now, having already defined $f^{m}\in\LipO{M_{m}}$ as
    \[ (\forall x\in\Omega_{m}^{x_{0}}) \quad f^{m}(x) = \lim_{k\to\infty} f^{m}_{\alpha_{m}(k)}(x), \]
    where $\alpha_{m}:\mathbb{N}\to\mathbb{N}$ is increasing, consider the sequence $(f^{m+1}_{\alpha_{m}(k)})_{k\geq m+1}$, which is well defined since $\alpha_{m}(k)\geq k \geq m+1$. Since this new sequence remains bounded in $\LipO{M_{m+1}}$, we can extract a subsequence which is pointwise convergent in $\LipO{M_{m+1}}$, that is, there exists $\alpha_{m+1}:\mathbb{N}\to\mathbb{N}$ increasing such that $(f^{m+1}_{\alpha_{m+1}(k)})_{k\geq m+1}$ converges pointwise, and define $f^{m+1}$ as its pointwise limit
    \[ (\forall x\in\Omega_{m+1}^{x_{0}}) \quad f^{m+1}(x) = \lim_{k\to\infty} f^{m+1}_{\alpha_{m+1}(k)}(x). \]
    We summarize now some properties of the functions defined this way. Notice that every $x\in\Omega$ belongs to $\Omega_{m}^{x_{0}}$ for every sufficiently big $m\in\mathbb{N}$. We claim that $m\mapsto f^{m}(x)$ is constant for every $m\in\mathbb{N}$ big enough. Let $m\in\mathbb{N}$ be such that $x\in\Omega_{m}^{x_{0}}$. Then, by definition, $f^{m}_{k}$ and $f^{m+1}_{k}$ coincide in $\Omega_{m}^{x_{0}}$ for $k>m$, which implies that $f^{m}(x) = f^{m+1}(x)$ for every $x\in\Omega_{m}^{x_{0}}$, since $\alpha_{m+1}$ defines a subsequence of $\alpha_{m}$. Then, we can define $f:\Omega\to\mathbb{R}$ as
    \[ f(x) = \lim_{m\to\infty} f^{m}(x). \]
    We will see that this function belongs to $\LipO{M}$. Indeed, let $x,y\in\Omega$ and take $m\in\mathbb{N}$ such that $x,y\in\Omega_{m}^{x_{0}}$. Then, noticing that $m\mapsto f^{m}(y)-f^{m}(x)$ is constant and that the functions $f^{m}$ are weak-$*$ limits of functions bounded by $\|g\|_{\infty}$ in $\LipO{M_{m}}$, we have that for every big enough $m\in\mathbb{N}$
    \[ |f(y)-f(x)| = |f^{m}(y)-f^{m}(x)| \leq \|g\|_{\infty} d_{m}(x,y), \]
    which leads to $|f(y)-f(x)| \leq \|g\|_{\infty} d(x,y)$, since $\displaystyle d(x,y)=\inf_{m\in\mathbb{N}}d_{m}(x,y)$, that is, $f\in\LipO{M}$ with $\lipnorm{f}\leq\|g\|_{\infty}$.

    We claim that $f\in\LipO{M}$ satisfies that for every $i\in\{1,\ldots,n\}$, $\partial_{i}f = g_{i}$ in $L^{\infty}(\Omega)$. Let $H=[a_{1},b_{1}]\times\cdots\times[a_{n},b_{n}]$ be a hyper-rectangle contained in $\Omega$, and denote its elements as $x = (z,t)$, where $z\in H'\subset\mathbb{R}^{n-1}$ is formed by all the coordinates of $x$ except for the $i^{\text{th}}$ coordinate, and $t=x_{i}$. Then, by Fubini's theorem
    \[ \int \partial_{i}f(x)\mathds{1}_{H}(x) dx = \int_{H'} \int_{a_{i}}^{b_{i}} \partial_{i}f(z,t) dt dz = \int_{H'} f(z,b_{i})-f(z,a_{i}) dz. \]
    Notice that for $m\in\mathbb{N}$ big enough, $H\subset\Omega_{m}^{x_{0}}$, from which we have that
    \[ \int_{H'} f(z,b_{i})-f(z,a_{i}) dz = \int_{H'} \lim_{k\to\infty} f^{m}_{\alpha_{m}(k)}(z,b_{i})-f^{m}_{\alpha_{m}(k)}(z,a_{i}) dz  \]
    \[ = \int_{H'} \left(\lim_{k\to\infty} \int_{0}^{1} (g*u_{\alpha_{m}(k)})(z,a_{i}+t(b_{i}-a_{i}))\cdot (b_{i}-a_{i})e_{i} dt \right) dz \]
    \[ = \int_{H'}\left(\lim_{k\to\infty} \int_{a_{i}}^{b_{i}} (g_{i}*u_{\alpha_{m}(k)})(z,t) dt\right) dz. \]
    Since $g_{i}*u_{\alpha_{m}(k)}\in L^{\infty}(\Omega)$, and thanks to Lemma \ref{restae}, the functions defined on $[a_{i},b_{i}]$ as $t\mapsto (g_{i}*u_{\alpha_{m}(k)})(z,t)$ belong to $L^{\infty}([a_{i},b_{i}])$ for almost every $z\in H'$, satisfy $(g_{i}*u_{\alpha_{m}(k)})(z,t)\leq\|g\|_{\infty}$ for almost every $t\in[a_{i},b_{i}]$ and converge almost everywhere on $t\in[a_{i},b_{i}]$ to $g_{i}(z,t)$, since $g_{i}*u_{\alpha_{m}(k)}$ converges almost everywhere in $\Omega$ to $g_{i}$. More precisely, recall that for a function $G:H=H'\times[a_{i},b_{i}]\to\mathbb{R}$, Lebesgue-measurability in $H$ implies that for almost every $z\in H'$, the function $t\mapsto G(z,t)$ is Lebesgue measurable in $[a_{i},b_{i}]$. Moreover, if $G$ is bounded up to a null measure set, thanks to Lemma \ref{restae}, the same is true for almost every $z\in H'$ for the function $t\mapsto G(z,t)$. By the same argument, we can deduce the convergence. From this, we finally see that
    \[ \int \partial_{i}f(x)\mathds{1}_{H}(x) dx = \int g_i(x)\mathds{1}_{H}(x) dx, \]
    that is, $\partial_{i}f = g_{i}$ almost everywhere in $\Omega$, from which we obtain $(1)$.

    Finally, we show that for every $g\in\textup{Im}(T)$, $\lipnorm{T^{-1}g}\leq\|g\|_{\infty}$, which together with Proposition \ref{Tbounded} implies that $T:\LipO{M}\to Y$ is an isometric isomorphism. Let $x,y\in\Omega$ and consider $\gamma_{x},\gamma_{y}:[0,1]\to\Omega$ Lipschitz paths going from $x_{0}$ to $x$ and $y$, respectively. If $\gamma_{xy}:[0,1]\to\Omega$ stands for the path given by rescaling the domain of the concatenation of the reverse of $\gamma_{x}$ and $\gamma_{y}$, we see that $\gamma_{xy}$ is a Lipschitz path going from $x$ to $y$. then
    \[ |T^{-1}g(y)-T^{-1}g(x)| \]
    \[ = \lim_{k\to\infty} \left| \int_{0}^{1} (g*u_{k})(\gamma_{y}(t))\cdot\gamma_{y}'(t) dt - \int_{0}^{1} (g*u_{k})(\gamma_{x}(t))\cdot\gamma_{x}'(t) dt \right| \]
    \[ = \lim_{k\to\infty}\left| \int_{0}^{1} (g*u_{k})(\gamma_{xy}(t))\cdot\gamma_{xy}'(t) dt \right| = \lim_{k\to\infty}\left| \int_{0}^{1} (g*u_{k})(\gamma(t))\cdot\gamma'(t) dt \right|, \]
    where $\gamma:[0,1]\to\Omega$ is any Lipschitz path going from $x$ to $y$. The last equality is justified by the fact that $g*u_{k}$ is conservative in $\Omega_{k}$ and that every Lipschitz path $\gamma$ going from $x$ to $y$ satisfies that $\gamma([0,1])$ is compact, which implies that $\gamma([0,1])\subset\Omega_{k}$ for every $k$ big enough. But notice that for a fixed path $\gamma$ as such and $k$ big enough
    \[ \left| \int_{0}^{1} (g*u_{k})(\gamma(t))\cdot\gamma'(t) dt \right| \leq \sup_{z\in\Omega_{k}} \|(g*u_{k})(z)\|_{*} \int_{0}^{1} \|\gamma'(t)\| dt \]
    \[ = \sup_{z\in\Omega_{k}} \|(g*u_{k})(z)\|_{*} \ell(\gamma). \]
    But considering that for $z\in\Omega_{k}$ it holds that
    \[ \|(g*u_{k})(z)\|_{*} = \left\| \int_{|y-x|<\frac{1}{k}} g(y)u_{k}(x-y) dy \right\|_{*} \]
    \[ \leq \int_{|y-x|<\frac{1}{k}} \|g(y)\|_{*} u_{k}(x-y) dy \leq \|g\|_{\infty} \int_{E} u_{k}(x-y) dy = \|g\|_{\infty}, \]
    we deduce that
    \[ \left| \int_{0}^{1} (g*u_{k})(\gamma(t))\cdot\gamma'(t) dt \right| \leq \sup_{z\in\Omega_{k}} \|(g*u_{k})(z)\|_{*} \ell(\gamma) \leq \|g\|_{\infty}\ell(\gamma), \]
    which leads to
    \[ |T^{-1}g(y)-T^{-1}g(x)| = \lim_{k\to\infty}\left| \int_{0}^{1} (g*u_{k})(\gamma(t))\cdot\gamma'(t) dt \right| \leq \|g\|_{\infty}\ell(\gamma). \]
    But $\gamma$ was arbitrary, so $|T^{-1}g(y)-T^{-1}g(x)| \leq \|g\|_{\infty}d(x,y)$. Hence, since the last inequality is valid for every $x,y\in\Omega$ and $g\in\textup{Im}(T)$, we have that $\lipnorm{T^{-1}g}\leq\|g\|_{\infty}$ for every $g\in\textup{Im}(T)$.
    
\end{proof}

Notice that Theorem \ref{main1} is written in the form of \cite[Proposition 3.2]{CKK17}, where $\Omega$ is an open convex subset of $E$. In order to reinforce this comparison, notice first that if $R\in\mathcal{D}'(\Omega)$ then, for every $\varphi\in\mathcal{D}(\Omega)$
\[ \langle \partial_{ij}R,\varphi \rangle = \langle R,\partial_{ji}\varphi \rangle = \langle R,\partial_{ij}\varphi \rangle = \langle \partial_{ji}R,\varphi \rangle, \]
where the second equality is valid thanks to the smoothness of $\varphi$. In particular, applying this for $f\in\LipO{M}$ and its associated distribution, we see that $g=\nabla f$ satisfies
\[ \langle \partial_{i}g_{j},\varphi \rangle = \langle \partial_{ij}f,\varphi \rangle = \langle \partial_{ji}f,\varphi \rangle = \langle \partial_{j}g_{i},\varphi \rangle \quad \forall \varphi\in\mathcal{D}'(\Omega). \]
Hence, $g\in Y \Rightarrow \partial_{i}g_{j}=\partial_{j}g_{i}$ in $\mathcal{D}'(\Omega)$, but the converse is not true in general. In the next proposition, we show that the converse becomes true under an extra hypothesis.

\begin{proposition}\label{simplyconn}
    Suppose now that $\Omega$ is a simply connected domain such that every connected component of $\Omega_{k}$ is simply connected. Then
    \[ Y = \{ g\in L^{\infty}(\Omega;E^{*}) : \partial_{i}g_{j} = \partial_{j}g_{i} \text{ in } \mathcal{D}'(\Omega) \} \]
\end{proposition}

\begin{proof}
    Since every $g\in Y$ satisfies $\partial_{i}g_{j} = \partial_{j}g_{i}$ in $\mathcal{D}'(\Omega)$, we just need to show the other inclusion. Suppose that $g\in L^{\infty}(\Omega;E^{*})$ is such that $\partial_{i}g_{j}=\partial_{j}g_{i}$ in $\mathcal{D}'(\Omega)$, and fix $k\in\mathbb{N}$ and $x\in\Omega_{k}$. We see that the function $\varphi(y) = u_{k}(x-y)$ belongs to $\mathcal{D}(\Omega_{k})$. Hence
    \[ \langle \partial_{i}g_{j},\varphi \rangle = \langle \partial_{j}g_{i},\varphi \rangle \iff \langle g_{j},\partial_{i}\varphi \rangle = \langle g_{i},\partial_{j}\varphi \rangle \]
    \[ \iff \int g_{j}(y)\partial_{i}\varphi(y) dy = \int g_{i}(y)\partial_{j}\varphi(y) dy. \]
    But we see that $\partial_{i}\varphi(y) = -\partial_{i}u_{k}(x-y)$ for every $i\in\{1,\ldots,n\}$. From this
    \[ \iff \int g_{j}(y)\partial_{i}u_{k}(x-y) dy = \int g_{i}(y)\partial_{j}u_{k}(x-y) dy \] 
    \[ \iff (g_{j}*\partial_{i}u_{k})(x) = (g_{i}*\partial_{j}u_{k})(x) \iff \partial_{i}(g_{j}*u_{k})(x) = \partial_{j}(g_{i}*u_{k})(x). \]
    By virtue of Lemma \ref{poincare}, this is equivalent to $g*u_{k}$ being conservative in every connected component of $\Omega_{k}$, since they are simply connected. Hence, $g*u_{k}$ is conservative.
\end{proof}

Notice that, in particular, convex domains satisfy the conditions of Proposition \ref{simplyconn}. Then, Proposition \ref{simplyconn} in combination with Theorem \ref{main1}, allows us to recover \cite[Proposition 3.2]{CKK17}.

\subsection{Identification of the Lipschitz-free space}

In Section \ref{Lip0}, we have established an isometry between $\LipO{M}$ and the subspace $Y$ of $L^{\infty}(\Omega;E^{*})$, consisting on vector fields whose mollifications are conservative vector fields. This isometry is given by the gradient operator defined over $\LipO{M}$. Our aim is to use our characterization for $Y$, which is the image of $T$, to show that $Y$ is actually the annihilator of a subspace of $L^{1}(\Omega;E)$, which is defined in terms of a differential operator in the sense of distributions. In order to do this, recall that if $V$ is a vector space and $U,W$ are subspaces of $V$ and $V^{*}$, respectively, the annihilator of $U$ and the preannihilator of $W$ are defined respectively as follows
\[ U^{\perp} := \{ z\in V^{*} : \langle z,u \rangle = 0 \text{ for every } u\in U \} \]
\[ W_{\perp} := \{ z\in V : \langle w,z \rangle=0 \text{ for every } w\in W \}\]
From now on, whenever $h\in L^{1}(\Omega;E)$, $\textup{div}(h)$ stands for the divergence in the sense of distributions in $E$ of the extension of $h$ by zero outside $\Omega$.

\begin{proposition}\label{Xprop}
    Let
    \[ X := \{ h \in L^{1}(\Omega;E) : \textup{div}(h) = 0 \text{ in } \mathcal{D}'(E) \}. \]
    Then
    \begin{enumerate}[label=\roman*)]
        \item\label{Xprop_1} $Y = X^{\perp}$ and $X = Y_{\perp}$ in the standard duality $(L^{1}(\Omega;E))^{*} = L^{\infty}(\Omega;E^{*})$.
        \item\label{Xprop_2} $X\cap C_{0}^{\infty}(\Omega;E)$ is dense in $X$.
    \end{enumerate}
\end{proposition}

\begin{proof}
    We begin by describing the outline of the proof. First we will prove that $Y = (X\cap \mathcal{C}_{0}^{\infty}(\Omega;E))^{\perp}$, from which we deduce that
    \[ Y = \overline{X\cap \mathcal{C}_{0}^{\infty}(\Omega;E)}^{\perp}\quad\text{and}\quad Y_{\perp} = \overline{X\cap \mathcal{C}_{0}^{\infty}(\Omega;E)}. \]
    Then, we proceed to show that $Y\subset X^{\perp}$, which allows us to deduce $\it{\ref{Xprop_2}}.$, which finally implies $\it{\ref{Xprop_1}}.$

    Let $g\in (X\cap \mathcal{C}_{0}^{\infty}(\Omega;E))^{\perp}$. We need to prove that $g*u_{k}$ is conservative for every $k$. To this end, fix $k\in\mathbb{N}$ and consider Lipschitz closed path $\gamma:[0,1]\to\Omega_{k}$. Then
    \[ \int_{\gamma} (g*u_{k})(\gamma(t))\cdot \gamma'(t) dt = \int_{0}^{1}\int_{\Omega} g(z)u_{k}(\gamma(t)-z)\cdot \gamma'(t) dzdt \]
    \[ = \int_{\Omega} \left[ \int_{0}^{1} u_{k}(\gamma(t)-z)\gamma'(t) dt \right] \cdot g(z) dz = \int_\Omega h(z)\cdot g(z) dz, \]
    where 
    \[ h(z) := \int_{0}^{1} u_{k}(\gamma(t)-z)\gamma'(t) dt. \]
    We claim that the function $h$ is smooth, compactly supported and has null divergence. Thanks to the regularity of $u_{k}$, it is clear that $h$ is smooth. Moreover, it is compactly supported, since if $z\notin \gamma([0,1]) + \overline{B}(0,1/k)$, then for every $t\in[0,1]$
    \[ \|\gamma(t)-z\| > \frac{1}{k}, \]
    from which $u_{k}(\gamma(t)-z)=0$ for every $t\in[0,1]$, which leads to $h(z)=0$. Then, $\supp(h)\subset \gamma([0,1]) + \overline{B}(0,1/k)$, which is compact. Finally, we see that
    \[ \textup{div}(h)(z) = \sum_{j=1}^{n}\int_{0}^{1} \partial_{j}u_{k}(\gamma(t)-z)\gamma'_{j}(t) dt \]
    \[ = \int_{0}^{1} \nabla u_{k}(\gamma(t)-z) \cdot \gamma'(t) dt = \int_{0}^{1} (u_{k}\circ(\gamma-z))'(t) dt \]
    \[ = u_{k}(\gamma(1)-z) - u_{k}(\gamma(0)-z) = 0. \]
    Then, $h\in X\cap \mathcal{C}_{0}^{\infty}(\Omega;E)$ and we deduce that
    \[ \int_{\gamma} (g*u_{k})(\gamma(t))\cdot \gamma'(t) dt = \int_{\Omega} h(z) \cdot g(z) dz = \langle g,h \rangle = 0. \]
    Then, for every closed Lipschitz curve $\gamma\subset\Omega_{k}$, $g*u_{k}$ integrates $0$ over $\gamma$, that is, $g*u_{k}$ is conservative by Proposition \ref{prop:cons}. But this is valid for every $k\in\mathbb{N}$, which leads to $g\in Y$, that is $(X\cap\mathcal{C}_{0}^{\infty}(\Omega;E))^{\perp}\subset Y$. 

    The proof of the other inclusion is exactly the same as in \cite{CKK17} and we give the details for completeness. Let $g = \nabla f$ for some $f\in\LipO{M}$ and choose $h\in X\cap \mathcal{C}_{0}^{\infty}(\Omega;E)$. Then
    \[ \langle g,h \rangle = \sum_{i=1}^{n} \langle \partial_{i} f,h_{i} \rangle = -\sum_{i=1}^{n} \langle f,\partial_{i}h_{i} \rangle = -\left\langle f, \textup{div}(h) \right\rangle = 0, \]
    that is, $g\in (X\cap \mathcal{C}_{0}^{\infty}(\Omega;E))^{\perp}$. From all this, we have that 
    \[ Y = (X\cap \mathcal{C}_{0}^{\infty}(\Omega;E))^{\perp} = \overline{X\cap \mathcal{C}_{0}^{\infty}(\Omega;E)}^{\perp}, \]
    and by virtue of Hahn-Banach theorem, 
    \[ \overline{X\cap \mathcal{C}_{0}^{\infty}(\Omega;E)} = Y_{\perp}.\]

    We claim that $Y\subset X^{\perp}$. To this end, consider the set $\{ \nabla \varphi : \varphi\in\mathcal{D}(\Omega) \}$. We claim that it is weak-$*$ dense in $Y$. To see this, we will prove that for every $f\in\LipO{M}$, there exists a sequence $(f_{k})_{k\in\mathbb{N}}\subset\mathcal{D}(\Omega)$ such that $\nabla f_{k}\overset{*}{\rightharpoonup} \nabla f$ in $L^{\infty}(\Omega;E^{*})$.

    First, we notice that the convolution $F_{m}:=f*u_{m}$ is smooth in $\Omega_{m}$ and satisfies
    \[
        \nabla F_{m} = (\nabla f)*u_{m} \quad\text{in }\Omega_{m}.
    \]
    For $R>0$ we build a cutoff that both keeps a fixed distance from the boundary and has the radial $1/R$ gradient decay. 
    Let $\eta:\mathbb{R}^{n}\to\mathbb{R}$ be a $\mathcal{C}^\infty$ function with $\eta\equiv1$ on $\overline{B}(0,1)$ and $\mathrm{supp}(\eta)\subset \overline{B}(0,2)$, and set $\eta_R(x):=\eta((x-x_0)/R)$. Then $\|\nabla\eta_R\|_\infty\leq C/R$, where
    \[ C=\sup_{y\in\mathbb{R}^{n}} \|\nabla\eta(y)\| \]
    and $\mathrm{supp}(\nabla\eta_R)\subset \overline{B}(x_0,2R)\setminus \overline{B}(x_0,R)$. 
    Let $d(x):=\operatorname{dist}(x,\Omega^c)$ and choose $\theta\in \mathcal{C}^\infty([0,\infty))$ with $\theta\equiv0$ on $[0,\tfrac12]$, $\theta\equiv1$ on $[1,\infty)$. Define $\tau_R(x):=\theta(R\,d(x))$. Then $\tau_R\equiv1$ on $\Omega_{1/R}$ and $\tau_R\equiv0$ outside $\Omega_{1/(2R)}$. 
    Set
    \[
        \chi_R:=\eta_R\,\tau_R.
    \]
    Hence $0\leq\chi_R\leq1$, $\chi_R\equiv1$ on $K_R:=\overline{B}(x_0,R)\cap\Omega_{1/R}$, and $\mathrm{supp}(\chi_R)\subset \overline{B}(x_0,2R)\cap\Omega_{1/(2R)}$. 
    Denote
    \[
      A_R:=\mathrm{supp}(\nabla\eta_R)\cap\Omega_{1/R}
      \qquad 
      S_R:=\left\{x\in\Omega:\frac{1}{2R}<d(x)<\frac{1}{R}\right\}.
    \]
    On $A_R$ we have $\tau_R\equiv1$ and $\nabla\tau_R=0$, so $\nabla\chi_R=\nabla\eta_R$ there and thus
    \[
      \sup_{y\in A_{R}}\|\nabla\chi_R(y)\|\leq \frac{C}{R}.
    \]
    Moreover, for $m\geq 2R$ we have $\mathrm{supp}(\chi_R)\subset\Omega_{1/(2R)}\subset\Omega_m$, so the products with $F_m$ extended by $0$ outside $\Omega_{m}$ remain $\mathcal{C}^{\infty}$.
    Define
    \[
        \phi_{R,m}:=\chi_{R}\,F_{m}\in\mathcal D(\Omega;E),
    \]
    from which we see that $\nabla\phi_{R,m}=\chi_{R}(\nabla f*u_{m}) + F_{m}\,\nabla\chi_{R}$. We claim that $\nabla\phi_{R,m}\stackrel{*}{\rightharpoonup} \nabla f$ along a diagonal sequence.
    For any $\varphi\in L^{1}(\Omega;E)$,
    \[
    \begin{aligned}
        \langle \nabla\phi_{R,m}-\nabla f,\varphi\rangle
        &= \underbrace{\langle \nabla f*u_{m}-\nabla f,\chi_{R}\varphi\rangle}_{A_{R,m}}
        + \underbrace{\langle \nabla f,(\chi_{R}-1)\varphi\rangle}_{B_{R}}
        + \underbrace{\langle F_{m}\nabla\chi_{R},\varphi\rangle}_{C_{R,m}}.
    \end{aligned}
    \]
    Let $\tilde{u}_{m}(x):= u_{m}(-x)$ be the conjugate of $u_{m}$. Since $\langle \nabla f*u_{m},\psi\rangle=\langle \nabla f,\psi*\tilde u_{m}\rangle$ and $\psi*\tilde u_{m}\to\psi$ in $L^{1}$ for any $\psi\in L^{1}$, we get $A_{R,m}\to0$ as $m\to\infty$ with $R$ fixed. Also, $\chi_{R}\to 1$ pointwise and $|\chi_{R}-1||\varphi|\leq |\varphi|\in L^{1}$, hence $B_{R}\to0$ as $R\to\infty$.

    For $C_{R,m}$, since $f$ is Lipschitz and $f(x_{0})=0$, we have that for every $x\in\Omega_{m}$
    \[
      |F_{m}(x)|\leq \int |f(x-y)|u_{m}(y) dy =  \int |f(x-y)-f(x_0)||u_{m}(y)|dy \]
    \[ \leq \int \Lip{f}\|x-y-x_0\|u_{m}(y) dy \leq \Lip{f} \int(\|x-x_0\|+\|y\|)u_{m}(y)dy \]
    \[ \leq \Lip{f}\left(\|x-x_0\| + \frac{1}{m}\right).\]
    Decompose using the sets above:
    \[
      C_{R,m}=\langle F_m\nabla\chi_R,\varphi\,\mathds{1}_{A_R}\rangle+\langle F_m\nabla\chi_R,\varphi\,\mathds{1}_{S_R}\rangle.
    \]
    On $A_R$ we have $\|\nabla\chi_R\|_\infty\le C/R$, hence
    \[
      \|F_m\nabla\chi_R\|_{L^\infty(A_R)}
      \leq \frac{C}{R}\sup_{y\in A_R}|F_m(y)|.
    \]
    But since $A_R\subset\supp(\nabla\eta)\subset \overline{B}(x_{0},2R)\setminus\overline{B}(x_0,R)$ and $m\geq 2R$,  we see that if $y\in A_R$
    \[ |F_m(y)|\leq \Lip{f}\left( 2R + \frac{1}{m} \right) \leq \Lip{f}\left(2R+\frac{1}{2R}\right). \]
    Hence, for $R\geq 1$
    \[
      \|F_m\nabla\chi_R\|_{L^\infty(A_R)}
      \leq \frac{C}{R}\Lip{f}\left(2R+\frac{1}{2R}\right) = C\Lip{f}\left(2+\frac{1}{2R^2}\right) \leq 3C\Lip{f}.
    \]
    and therefore $|\langle F_m\nabla\chi_R,\varphi\,\mathds{1}_{A_R}\rangle|\leq 3C\Lip{f}\|\varphi\,\mathds{1}_{A_R}\|_{L^1}\to 0$ as $R\to\infty$.
    On $S_R$ the strip thickness is $\sim 1/R$, so $|S_R|\to0$ and $\|\varphi\,\mathds{1}_{S_R}\|_{L^1}\to0$ by absolute continuity of the integral; moreover $\|F_m\nabla\chi_R\|_{L^\infty(S_R)}$ is bounded uniformly in $m\ge 2R$ by the previous estimate of $F_m$ and the construction of $\chi_R$. Hence $|\langle F_m\nabla\chi_R,\varphi\,\mathds{1}_{S_R}\rangle|\to0$ as $R\to\infty$. Altogether, $\sup_{m\ge 2R}|C_{R,m}|\to0$.

    Given $\varepsilon>0$, choose $R$ large so that $|B_{R}|+\sup_{m\ge 2R}|C_{R,m}|<\varepsilon/2$, and then $m\ge 2R$ large with $|A_{R,m}|<\varepsilon/2$. Thus $\langle \nabla\phi_{R,m}-\nabla f,\varphi\rangle\to0$. Picking a diagonal sequence $(R_{k},m_{k})$ we obtain $\phi_{k}:=\phi_{R_{k},m_{k}}\in\mathcal D(\Omega;E)$ with
    \[
        \nabla\phi_{k}\overset{*}{\rightharpoonup} \nabla f \quad \text{in } L^{\infty}(\Omega;E^{*}).
    \]
    Setting $f_{k}:=\phi_{k}$ proves the desired weak-$*$ density of $\{\nabla\varphi:\varphi\in\mathcal D(\Omega)\}$ in $Y$.

    Returning to the original objective, which was to prove that $Y\subset X^{\perp}$, suppose that $g=\nabla f$ for some $f\in\LipO{M}$ and consider a sequence $(f_{k})_{k\in\mathbb{N}}$ in $\mathcal{D}(\Omega)$ such that $\nabla f_{k} \overset{*}{\rightharpoonup} g$ in $L^{\infty}(\Omega;E^{*})$. Then, for every $h\in X$
    \[ \langle g,h \rangle = \lim_{k\to\infty} \langle \nabla f_{k}, h \rangle. \]
    Since the extension by $0$ outside $\Omega$ of $f_{k}$ is smooth and compactly supported, it belongs to $\mathcal{D}(E)$. This implies that $\langle \nabla f_{k},h \rangle = \langle f_{k},-\textup{div}(h) \rangle = 0$ for every $k\in\mathbb{N}$, which leads to $\langle g,h \rangle = 0$, that is, $g\in X^{\perp}$.

    Recall that $Y=\overline{X\cap\mathcal{C}_{0}^{\infty}(\Omega;E)}^{\perp}$. This together with $Y\subset X^{\perp}$ implies that $\overline{X\cap\mathcal{C}_{0}^{\infty}(\Omega;E)}^{\perp}\subset X^{\perp}$, which leads to $\overline{X} = \overline{X\cap\mathcal{C}_{0}^{\infty}(\Omega;E)}$. Then, to obtain $\it{\ref{Xprop_2}}.$, we need to see that $X$ is closed.

    Suppose that a sequence $(h_{k})_{k\in\mathbb{N}}\subset X$ converges to $h\in L^{1}(\Omega;E)$. Since norm convergence implies weak convergence and $\nabla \varphi \in L^{\infty}(E;E^{*})$ for every $\varphi\in\mathcal{D}(E)$, we see that
    \[ \langle -\textup{div}(h),\varphi \rangle = \langle h,\nabla\varphi \rangle = \lim_{k\to\infty} \langle h_{k},\nabla\varphi \rangle = \lim_{k\to\infty} \langle -\textup{div}(h_{k}),\varphi \rangle = 0, \]
    which proves that $X$ is closed.

    Hence, $X=\overline{X} = \overline{X\cap\mathcal{C}_{0}^{\infty}(\Omega;E)}$, that is, $X\cap\mathcal{C}_{0}^{\infty}(\Omega;E)$ is dense in $X$, and
    \[ Y = \overline{X\cap \mathcal{C}_{0}^{\infty}(\Omega;E)}^{\perp} = X^{\perp} \quad\text{and}\quad  Y_{\perp} = \overline{X\cap \mathcal{C}_{0}^{\infty}(\Omega;E)} = X, \]
    which finishes the proof.
\end{proof}

\begin{remark}
    The fact that in the definition of $X$ the divergence $\textup{div}(h)$ is considered in the sense of distributions over $E$ instead of $\Omega$ is of utmost importance, since depending on that, the obtained sets might be different. This fact can be shown with an example in the case $E=\mathbb{R}$. Suppose that $\Omega = (-1,1)\subset \mathbb{R}$ and choose $0$ as the base point for the metric space $(M,d)$. Since $\Omega$ is convex, $d$ coincides with the usual metric of $\mathbb{R}$ restricted to $\Omega$. In this case, the set 
    \[ X' := \{ h \in L^{1}((-1,1)) : h' = 0 \text{ in } \mathcal{D}'((-1,1)) \} \]
    contains the constants functions, while on the other hand, $X=\{0\}$ thanks to connectedness. Notice that if Proposition \ref{Xprop} were true for $X'$ instead of $X$, we would have that
    \[ (L^{1}((-1,1))/X')^{*} \equiv X'^{\perp} = Y \equiv \LipO{(-1,1)}. \]
    Then, considering the function $f(t)=t$ (which belongs to $\LipO{(-1,1)}$) we would have that for every constant function on $(-1,1)$ (say, $h(x)=a$)
    \[ 0 = \langle f',h \rangle = a\int_{-1}^{1} dt = 2a, \]
    which is clearly a contradiction.

    More generally, suppose that $\Omega\subset E$ is a bounded domain with Lipschitz boundary. By the same argument, considering
    \[ X' := \{ h \in L^{1}(\Omega;E) : \textup{div}(h) = 0 \text{ in } \mathcal{D}'(\Omega) \}, \]
    if Proposition \ref{Xprop} were true for $X'$ instead of $X$, we would have that
    \[ L^{1}(\Omega;E)/X' \equiv X'^{\perp} = Y \equiv \LipO{M}.\]
    Then, for every $f\in\LipO{M}\cap\mathcal{C}^{\infty}(\Omega)$ and $h(x)=v$ constant in $\Omega$ we would have
    \[ 0 = \langle \nabla f,h \rangle = \int_{\Omega} \nabla f(x)\cdot v dx. \]
    But noticing that
    \[ \textup{div}(fv)(x) = \nabla f(x) \cdot v, \]
    by virtue of Gauss' Theorem \cite[Theorem 37.22]{LM05} we would have that
    \[ 0 = \int_{\Omega} \nabla f(x)\cdot v dx = \int_{\partial\Omega} fv \cdot ndS \quad \forall v\in E. \]
    Notice that, since $\Omega$ has Lipschitz boundary, $\overline{\Omega}$ coincides with the completion of $\Omega$ with respect to $d$, which implies that $f$ can be continuously extended to $\partial \Omega$, so the integral for the flux of $fv$ over $\partial\Omega$ is well defined. In other words, the net flux of the field $fv$ is $0$, whenever $f\in\LipO{M}$ is smooth and $v\in E$, which is certainly false in general.
\end{remark}

Recall the following classical result on Banach spaces.

\begin{theorem*}\cite[\textbf{Theorem 4.9}]{R91}
    Let $U$ be a closed subspace of a Banach space $V$.
    \begin{enumerate}
        \item The Hahn-Banach theorem extends each $u^{*}\in U^{*}$ to $v^{*}\in V^{*}$. Define
        \[ \sigma u^{*} = v^{*}+U^{\perp}. \]
        Then $\sigma$ is an isometric isomorphism of $U^{*}$ onto $V^{*}/U^{\perp}$.
        \item Let $\pi:V\to V/U$ be the quotient map. Put $W=V/U$. For each $w^{*}\in W^{*}$, define
        \[ \tau w^{*} = w^{*}\circ\pi. \]
        Then $\tau$ is an isometric isomorphism of $W^{*}$ onto $U^{\perp}$
    \end{enumerate}
\end{theorem*}

Considering this representations for the duals of quotients and subspaces, we proceed to study the adjoint operator of $T$. More precisely, since $T$ is an isometry from $\LipO{M}$ onto $Y$, $T^{*}$ becomes an isometry from $Y^{*}$ onto $(\LipO{M})^{*}$ (see e.g. \cite[Theorem 4.10]{R91}). In the following, we study the behavior of $T^{*}$ restricted to $L^{1}(\Omega;E)/X$, considered as a closed subspace of $Y^{*}\equiv X^{\perp} \subset (L^{\infty}(\Omega;E^{*}))^{*}$, which finally leads us to the second main result of this work.

\begin{theorem}\label{main2}
    Let $\Omega\subset E$ be a domain and consider $M$ the set $\Omega$ endowed with its intrinsic distance and fix $x_{0}\in\Omega$ as the base point of $M$. Then, $\lipfree{M}$ is linearly isometric to $L^{1}(\Omega;E)/X$, where
    \[ X := \{ h\in L^{1}(\Omega;E) : \textup{div}(h) = 0 \text{ in } \mathcal{D}'(E) \}. \]
    Moreover, if $S$ is the preadjoint of the linear isometry $T$ in Theorem \ref{main1}, then $S[h] = \delta(x)$ if and only if $-\textup{div}(h) = \delta_{x}-\delta_{x_{0}}$ in $\mathcal{D}'(\Omega)$, where for $x\in\Omega$, $\delta_{x}$ is the Dirac distribution centered at $x$ and $[h]$ stands for the equivalence class of $h$ in $L^{1}(\Omega;E)/X$
\end{theorem}

\begin{proof}
    This proof is divided in several stages. We start by showing that $(T^{*})^{-1}$ maps $\lipfree{M}\subset(\LipO{M})^{*}$ isometrically onto a closed subspace of $L^{1}(\Omega;E)/X\subset (L^{\infty}(\Omega;E^{*}))^{*}/Y^{\perp}$. Then, we prove a similar result for $T^{*}$, that is, that $T^{*}$ maps $L^{1}(\Omega;E)/X\subset (L^{\infty}(\Omega;E^{*}))^{*}/Y^{\perp}$ isometrically onto a closed subspace of $\lipfree{M}\subset(\LipO{M})^{*}$. From this, we get the first part of Theorem \ref{main2}, which says that $\lipfree{M}$ is linearly isometric to $L^{1}(\Omega;E)/X$. Then, by noticing that $S:=\restricted{T^{*}}{L^{1}(\Omega;E)/X}$ satisfies $S^{*}=T$, we show the characterization of evaluation functionals in terms of the operator $S$.

    \begin{enumerate}
   
        \item $(T^{*})^{-1}(\lipfree{M})\subset L^{1}(\Omega;E)/X$: Recall that (by definition) if $x_{0}\in\Omega$ is the base point for $M$, then $\delta(x_{0})=0$ in $(\LipO{M})^{*}$. From this, given $x\in\Omega$ and a piecewise linear path with nodes $x_{0}=y_{1},\ldots,y_{k}=x\in\Omega$, by linearity of $(T^{*})^{-1}$ we have
        \[ (T^{*})^{-1}(\delta(x)) = \sum_{i=1}^{k-1} (T^{*})^{-1}(\delta(y_{i+1})-\delta(y_{i})). \]
        Hence, it suffices to show that for any $x,y\in\Omega$ such that the segment $[x,y]$ is contained in $\Omega$, there exists $h\in L^{1}(\Omega;E)$ such that $(T^{*})^{-1}(\delta(y)-\delta(x))=[h]$  in $(L^{\infty}(\Omega;E^{*}))^{*}/Y^{\perp}$.

        Consider then $x,y\in\Omega$ such that $[x,y]\subset\Omega$. By virtue of compactness of this set, there exists $\varepsilon>0$ such that $[x,y]+\varepsilon \overline{B}\subset\Omega$. Define $A=\varepsilon\overline{B}\cap \{y-x\}^{\perp}$, that is, the closed ball of $E$ centered at the origin with radius $\varepsilon$ restricted to the subspace orthogonal to $y-x$. Let $\psi:[0,1]\to\mathbb{R}$ be a Lipschitz function satisfying the following conditions
        \begin{itemize}
            \item $\psi(0)=\psi(1)=0$, and
            \item $0<\psi\leq 1$ in $(0,1)$.
        \end{itemize}
        Define the set $H$ as
        \[ H:=\{ x+t(y-x) + \psi(t)u : t\in[0,1], u\in A \} \]
        Notice that $H\subset\Omega$ by virtue of the choice of $\varepsilon$ and $\psi$. Let $b_{1},\ldots,b_{n-1}\in\mathbb{R}^{n}$ be an orthogonal basis of $\{y-x\}^{\perp}$ and $b_{n}=y-x$. Consider the change of variables $r:\mathbb{R}^{n}\to\mathbb{R}^{n}$ given by 
        \[ r(u) := x+u_{n}b_{n}+\sum_{k=1}^{n-1}u_{k}b_{k}, \]
        which satisfies $|\det(Dr(u))|=|b_{1}|\cdots |b_{n}|$ (denoted by $B>0$ from now on) by virtue of orthogonality. If $S:=\{ u\in\mathbb{R}^{n-1} : \sum_{k=1}^{n-1} u_{k}b_{k} \in A \}$, we see that
        \[ r^{-1}(H) = \{ (v,t)\in\mathbb{R}^{n-1}\times\mathbb{R} : t\in[0,1] , v\in\psi(t)S \}. \]
        By virtue of the change of variables formula (see e.g. \cite[Theorem 3.9]{EG12}) we have that for every integrable function $h:\Omega\to\mathbb{R}$
        \[ \int_{\Omega} h(z) dz = \int_{r^{-1}(\Omega)} h(r(u)) |\det(Dr(u)) | du = B\int_{r^{-1}(\Omega)} h(r(u)) du.\]
        In particular, if $h$ is supported in $H$, and by virtue of Fubini's theorem
        \[ \begin{aligned}
            \int_{H} h(z) dz &= B\int_{r^{-1}(H)} h(r(u))  du \\
            &= B \int_{0}^{1}\int_{\psi(t)S} h(r(v,t)) dv dt \\
            &= B \int_{0}^{1} \int_{S} h\big(r(\psi(t)w,t)\big) \psi(t)^{n-1} dwdt,
        \end{aligned} \]
        where the last equality follows from the change of variables $w=\psi(t)v$ in $\mathbb{R}^{n-1}$, which is valid for every $t\in(0,1)$. In particular,
        \[ \begin{aligned}
            \int_{H} dz &= B\int_{0}^{1}\int_{S} \psi(t)^{n-1} dw dt \\
            &= B|S|\int_{0}^{1}\psi(t)^{n-1}dt,
        \end{aligned}  \]
        where $|S|$ is the volume of $S$ in $\mathbb{R}^{n-1}$. Since $\psi$ is continuous in $[0,1]$, we deduce that $H$ has finite volume.

        Let $I\subset(0,1)$ be the set where $\psi$ is differentiable and define the function $h_{xy}:\Omega\to E$ as
        
        \[ h_{xy}(z):= \begin{cases}
            \displaystyle\frac{y-x+\psi'(t)w}{B|S|\psi(t)^{n-1}} &, z=r(\psi(t)w,t) \text{ for some } w\in S, t\in I \\ 0&,\text{otherwise}
        \end{cases} \]
        We see that $h_{xy}$ is well defined, since whenever $z\in H\setminus\{x,y\}$, there exists a unique pair $(w,t)\in S\times(0,1)$ such that $z=r(\psi(t)w,t)$. Notice that $h_{xy}\in L^{1}(\Omega;E)$, since it is direct from its definition that it is measurable, supported in $H$ and, recalling that $|\psi'(t)|$ is bounded in $[0,1]$ by some $M>0$ as $\psi$ is Lipschitz, we have that
        \begin{align*}
            \int_{H}\|h_{xy}(z)\|dz &= B\int_{0}^{1}\int_{S} \|h_{xy}\big(r(\psi(t)w,t)\big)\| \psi(t)^{n-1} dwdt \\
            &= \frac{1}{|S|}\int_{0}^{1}\int_{S}\|y-x+\psi'(t)w\|dwdt \\
            &\leq \frac{1}{|S|}\int_{0}^{1}\int_{S}\|y-x\|+|\psi'(t)|\|w\|dwdt \\
            &\leq \frac{1}{|S|}\int_{0}^{1}\int_{S}\|y-x\|+M\varepsilon dwdt \\
            & = \|y-x\|+M\varepsilon.
        \end{align*}


        We claim that $(T^{*})^{-1}(\delta(y)-\delta(x)) = [h_{xy}]$. Let $g=\nabla f$ for some $f\in\LipO{M}$. By definition of $(T^{*})^{-1}$ we have that
        \[ \langle (T^{*})^{-1}(\delta(y)-\delta(x)), g \rangle = \langle \delta(y)-\delta(x) , T^{-1}(\nabla f) \rangle \]
        \[ = \langle \delta(y)-\delta(x), f \rangle = f(y)-f(x). \]
        On the other hand, by noticing that for fixed $w\in S$
        \[ \frac{d}{dt}\big( r(\psi(t)w,t) \big) = y-x+\psi'(t)w, \]
        we have by virtue of Fubini's theorem that
        \[
            \begin{aligned}
            \langle [h_{xy}],g \rangle & = \int_{\Omega} h_{xy}(z)\cdot g(z) dz = \int_{H} h_{xy}(z)\cdot\nabla f(z) dz\\
            &= B\int_{S} \int_{0}^{1} h_{xy}\big(r(\psi(t)w,t)\big)\cdot\nabla f\big(r(\psi(t)w,t\big) \psi(t)^{n-1} dt dw\\
            &= \frac{1}{|S|}\int_{S}\int_{0}^{1}\nabla f\big(r(\psi(t)w,t)\big) \cdot \frac{d}{dt}\big(r(\psi(t)w,t)\big) dt dw \\
            &= \frac{1}{|S|}\int_{S} f(y)-f(x) dw = f(y)-f(x).
            \end{aligned}
        \]
        But noticing that for fixed $w\in S$, $\gamma_{w}(t):=r(\psi(t)w,t)=x+t(y-x)+\varphi(t)w$ is a Lipschitz path going from $x$ to $y$, we obtain that
         \[
            \langle [h_{xy}],g \rangle = \frac{1}{|S|}\int_{S} f(y)-f(x) du = f(y)-f(x),
        \]
        which finally implies that $(T^{*})^{-1}(\delta(y)-\delta(x))=[h_{xy}]$, since the integration is valid for every $g=\nabla f$. Taking into account the remark at the beginning of this part of the proof, we see that for every $x\in\Omega$
        \[ (T^{*})^{-1}(\delta(x)) = \left[\sum_{i=1}^{k-1} h_{y_{i}y_{i+1}}\right]. \]
        Since $\textup{span}\{ \delta(x) : x\in M \}$ is dense in $\lipfree{M}\subset (\LipO{M})^{*}$ and $T^{*}$ is a bijective isometry, we deduce that $(T^{*})^{-1}$ maps $\lipfree{M}$ isometrically onto
        \[ \overline{\textup{span}}\{ [h_{xy}] : x,y\in M, [x,y]\subset\Omega \}, \]
        which is a closed subspace of $L^{1}(\Omega;E)/X \subset (L^{\infty}(\Omega;E^{*}))^{*}/Y^{\perp}$ by noticing that $[h_{xy}]\cap L^{1}(\Omega;E) \in L^{1}(\Omega;E)/X$.
        
        \item $T^{*}(L^{1}(\Omega;E)/X)\subset \lipfree{M}$: Since simple integrable functions are dense in $L^{1}(\Omega;E)$, it suffices to show that $T^{*}([\mathds{1}_{C}e_{i}])\in\lipfree{M}\subset(\LipO{M})^{*}$ whenever $C=[a_{1},b_{1}]\times\ldots\times[a_{n},b_{n}]$ is an hyper-rectangle contained in $\Omega$ and $e_{i}$ is the $i^{\text{th}}$ element of the canonical basis of $E$.

        In this setting, denote by $C_{i}$ the product of the intervals $[a_{k},b_{k}]$, except for $[a_{i},b_{i}]$. We see that if $f\in\LipO{M}$, then
        \[ \langle S([\mathds{1}_{C}e_{i}]),f \rangle = \langle \mathds{1}_{C}e_{i},\nabla f \rangle = \int_{C} \partial_{i}f(z) dz = \int_{C_{i}}\int_{a_{i}}^{b_{i}} \partial_{i}f(u,t) dtdu \]
        \[ = \int_{C_{i}} f(u,b_{i}) - f(u,a_{i}) du \]
        Define $F:C_{i}\to\lipfree{M}$ by $F(u) = \delta(u,b_{i})-\delta(u,a_{i})$. Notice that $F$ is separably-valued (since $\lipfree{M}$ is separable) and weakly-measurable (since for every $f\in\LipO{M}$ the function $\langle f,F(u) \rangle = f(u,b_{i})-f(u,a_{i})$ is continuous in $C_{i}$). Then, by virtue of Pettis's Measurability Theorem \cite[Theorem II.1.2]{DU77} we have that $F$ is measurable. Moreover, thanks to \cite[Theorem II.2.2]{DU77} it is Bochner integrable, since
        \[ \int_{C_{i}} \|F(u)\| du = \int_{C_{i}} \|\delta(u,b_{i})-\delta(u,a_{i})\| du \]
        \[ = \int_{C_{i}} d((u,b_{i}),(u,a_{i})) du = |b_{i}-a_{i}|\|e_{i}\| \int_{C_{i}} du \]
        \[ = |b_{i}-a_{i}||C_{i}|\|e_{i}\| = |C|\|e_{i}\| < \infty. \]
        Then, by virtue of Hille's Theorem \cite[Theorem II.2.6]{DU77}, we deduce that for every $f\in\LipO{M}$
        \[ \langle S([\mathds{1}_{C}e_{i}]),f \rangle = \left\langle \int_{C_{i}} \delta(u,b_{i})-\delta(u,a_{i}) du , f \right\rangle. \]
        In other words, $S([\mathds{1}_{C}e_{i}])$ is the Bochner integral of the difference between the values of the evaluation functional $\delta$ on the faces of $C$ at $x_{i}=b_{i}$ and $x_{i}=a_{i}$, which belongs to $\lipfree{M}$.

        Since $\textup{span}\{ [\mathds{1}_{C}e_{i}] : C\subset\Omega \text{ hyper-rectangle}, 1\leq i \leq n \}$ is dense in $L^{1}(\Omega;E)/X \subset (L^{\infty}(\Omega;E^{*}))^{*}/Y^{\perp}$ and $T^{*}$ is a bijective isometry, we deduce that $T^{*}$ maps $L^{1}(\Omega;E)/X$ isometrically onto
        \[ \overline{\textup{span}}\left\{ \int_{C_{i}} \delta(u,b_{i})-\delta(u,a_{i}) du : C\subset\Omega \text{ hyper-rectangle}, 1\leq i \leq n \right\}, \]
        which is a closed subspace of $\lipfree{M}\subset (\LipO{M})^{*}$.

        \item $L^{1}(\Omega;E)/X \equiv \lipfree{M}$: Consider $S:L^{1}(\Omega;E)/X\to\lipfree{M}$ given by the restriction of $T^{*}$ to $L^{1}(\Omega;E)/X$. By virtue of step $2$, $S$ is well defined. Moreover, by virtue of step $1$, $S$ is surjective. But since $T$ is a bijective isometry, so is $T^{*}$, which implies that $S$ is a linear isometry. In particular, $S$ is also injective. This shows that $S$ defines an isometric isomorphism between $L^{1}(\Omega;E)/X$ and $\lipfree{M}$. Moreover, for $f\in\LipO{M}$ and $[h]\in L^{1}(\Omega;E)/X$
        \[ \langle S^{*}f,[h] \rangle = \langle f,S[h] \rangle = \langle f,T^{*}[h] \rangle = \langle Tf,[h] \rangle, \]
        that is, $S^{*}=T$, or in other words, $S$ is the preadjoint of $T$.
        \item $S[h]=\delta(x) \Leftrightarrow -\textup{div}(h)=\delta_{x}-\delta_{x_{0}}$: In $(1)$, we shown that
        \[ (T^{*})^{-1}(\delta(x)) = \left[\sum_{i=1}^{k-1} h_{y_{i}y_{i+1}}\right], \]
        where $y_{i}\in\Omega$ are the nodes of a piecewise linear path going from $x_{0}$ to $x$.
        This is equivalent to
        \[ S\left[\sum_{i=1}^{k-1} h_{y_{i}y_{i+1}}\right] = \delta(x). \]
        Then, in order to prove that $S[h]=\delta(x) \Leftrightarrow -\textup{div}(h)=\delta_{x}-\delta_{x_{0}}$, it suffices to show that
        \[ -\textup{div}\left( \sum_{i=1}^{k-1} h_{y_{i}y_{i+1}} \right) = \delta_{x}-\delta_{x_{0}}. \]
        The reason for this is that
        \[ S[h] = \delta(x) \iff [h]=\left[\sum_{i=1}^{k-1} h_{y_{i}y_{i+1}}\right] \]
        \[ \iff \textup{div}\left(h-\sum_{i=1}^{k-1} h_{y_{i}y_{i+1}}\right)=0 \iff -\textup{div}(h) = -\textup{div}\left(\sum_{i=1}^{k-1} h_{y_{i}y_{i+1}}\right). \]

        For $\varphi\in\mathcal{D}(E)$, we see that $f:= \varphi - \varphi(x_{0})$ restricted to $\Omega$ belongs to $\LipO{M}$ and $\nabla f = \nabla \varphi$ in $E$. Hence, noticing that each $h_{y_{i}y_{i+1}}$ is compactly supported in $\Omega$ implies that their sum also is compactly supported in $\Omega$, we see that
        \[ \left\langle -\textup{div}\left(\sum_{i=1}^{k-1} h_{y_{i}y_{i+1}}\right),\varphi \right\rangle = \left\langle \sum_{i=1}^{k-1} h_{y_{i}y_{i+1}},\nabla\varphi \right\rangle = \sum_{i=1}^{k} \langle h_{x_{i-1}x_{i}},\nabla f \rangle  \]
        \[ = \sum_{i=1}^{k} f(x_{i})-f(x_{i-1}) = \sum_{i=1}^{k} \varphi(x_{i})-\varphi(x_{i-1}) \]
        \[ = \varphi(x)-\varphi(x_{0}) = \langle\delta_{x}-\delta_{x_{0}},\varphi \rangle. \]
        Hence,
        \[ -\textup{div}\left( \sum_{i=1}^{k-1} h_{y_{i}y_{i+1}} \right) = \delta_{x}-\delta_{x_{0}}, \]
        which concludes the proof.
    \end{enumerate}

\end{proof}

\begin{remark}
    In \cite{CKK17}, Proposition 3.4 states that when $\Omega\subset E$ is a convex domain, there exists a compactly supported class representative $h\in L^{1}(\Omega;E)$ such that $\textup{div}(h)=\delta_{x_{0}}-\delta_{x}$. In the proof of Theorem \ref{main2}, the exhibited class representative given by
    \[ \overline h := \sum_{i=1}^{k-1} h_{y_{i}y_{i+1}} \]
    is compactly supported, being the sum of finitely many compactly supported functions, which can even be defined in a way such that their supports are pairwise disjoint. 
    
    Notice also that in step $1$ of the proof of Theorem \ref{main2}, the function $h_{xy}$ associated to $x,y\in\Omega$ such that $[x,y]\subset\Omega$ is essentially defined by means of an $(n-1)$-dimensional disk with a fixed radius oriented orthogonally to the segment and a function which changes that radius when going from $x$ to $y$. The conditions for $\psi$ were given in full generality, but we could have even chosen a regular function, such as $\psi(t)=t(1-t)$. The same construction can be made changing the disk for any other flat $(n-1)$-dimensional surface, as long as it has positive measure and $H$ remains contained in $\Omega$. In other words, for every $x\in\Omega$ and any piecewise linear curve going from $x_{0}$ to $x$, it is possible to find a class representative of $S^{-1}\delta(x)$ which is not only compactly supported, but such as its support is as close as we want to the chosen curve.

    Notice that the only crucial fact in the construction of these functions is that their support must in some way ``collapse" at both $x$ and $y$, and their values ``flow" from $x$ to $y$, being this flow solenoidal.
\end{remark}

\section{Final comments} \label{section4}

In what follows, we state the main results from \cite{BCJ05} and \cite{CKK17}, which can be viewed as particular cases of Theorem \ref{main2}. More precisely, we focus on \cite[Theorem 2.4]{BCJ05} and \cite[Theorem 1.1]{CKK17}.

First, we refer to \cite{BCJ05}. Let $K=\overline{\Omega}$ where $\Omega$ is a bounded connected open subset of $E$ with Lipschitz boundary. $K$ is endowed with the extension of its intrinsic distance to its completion. Notice that thanks to the Lipschitz boundary condition, the completion of $\Omega$ endowed with its intrinsic distance is simply its closure in $E$ with its correspondent intrinsic distance. Define $T_{\mu}$, for $\mu\in\lipfree{K}$, as the distribution given by
\[ \langle T_{\mu}, \varphi \rangle = \langle \varphi,\mu \rangle \quad \forall \varphi\in C^{\infty}(K). \]

\begin{theorem*}\cite[\textbf{Theorem 2.4}]{BCJ05}
    The following equality holds between subsets of $\mathcal{D}'(K)$ (the distributions over $K$):
    \[ \left\{ T_{\mu} : \mu\in\lipfree{K} \right\} = \left\{ -\textup{div}(h) : h\in L^{1}(\Omega;E) \right\}. \]
    Furthermore, if $X$ denotes the closed subspace
    \[ X:= \left\{ h\in L^{1}(\Omega;E) : \textup{div}(h) = 0 \right\}, \]
    the linear map $h\in L^{1}(\Omega;E)/X \mapsto -\textup{div}(h) \in \lipfree{K}$ is an isometry, i.e.:
    \[ \|h\|_{L^{1}(\Omega;E)/X} = \|\textup{div}(h)\|. \]
\end{theorem*}

We see that the equality in the space of distributions is actually the same obtained in Theorem \ref{main2} thanks to the definition of $T_{\mu}$ and the fact that $K$ is compact. In this framework, the requirement of Lipschitz regularity for the boundary of $\Omega$ assures that $K$ remains embedded in $E$. An example where this is not true is the following

\begin{example} 
    Consider the open bounded subset of $\mathbb{R}^{2}$ endowed with $\|\cdot\|_{1}$ given by $\Omega = B(0,1)\setminus \{ (t,0) : t\geq 0 \}$. Its closure is clearly $\overline{B}(0,1)$. But notice that the intrinsic distance in $\Omega$ is given by
    \[ d(x,y) = \left\{ \begin{array}{ccl} \|y-x\|_{1} &,& [x,y]\subset\Omega \\ \|x\|_{1}+\|y\|_{1} &,& [x,y]\not\subset\Omega \end{array} \right. \]
    Notice that this metric is not equivalent to the restriction to $\Omega$ of the metric induced by the norm. To see this, it suffices to consider the points in $\Omega$ $x_{t}=(1/2,t)$ and $y_{t}=(1/2,-t)$ for $t\in(0,1)$. Clearly, $[x_{t},y_{t}]\not\subset \Omega$, from which $d(x_{t},y_{t})=1+2t$. On the other hand, $\|x_{t}-y_{t}\|_{1}=2t$, from which we see that as $t\to 0$
    \[ \frac{d(x_{t},y_{t})}{\|x_{t}-y_{t}\|_1} = \frac{1+2t}{2t} \to \infty. \]
    With a similar procedure, we can obtain actually exhibit a function that is Lipschitz in $(\Omega,d)$ but not in $(\Omega,\|\cdot\|_{1})$. More precisely, this function is given by the distance to $(1^{-},0^{+})$ within $\Omega$, that is
    \[ f(x) = \left\{ \begin{array}{ccl} 1-x_{1}+x_{2}&,& x_{2}\geq 0 \\ 1+|x_{1}|-x_{2}&,& x_{2}<0 \end{array} \right. \]
    Notice that in this example, although $\Omega$ is open and bounded, it does not have Lipschitz boundary, which implies that there are Lipschitz functions in $(\Omega,d)$ that cannot be extended to the closure of $\Omega$. Nevertheless, these functions can be extended to the completion of $\Omega$ with respect to $d$. This completion is not embedded in $\mathbb{R}^{2}$ since the metrics are not equivalent, but it can be easily shown that the aforementioned completion (as a metric space) is isometric to the following subset of $\mathbb{R}^{3}$ endowed with $\|\cdot\|_{1}$
    \[ M_{\Omega} = \overline{B}(0,1)\cap(\{ (x,y,0) : x\leq 0 \} \cup \{ (0,y,z) : yz\leq 0 \}). \]

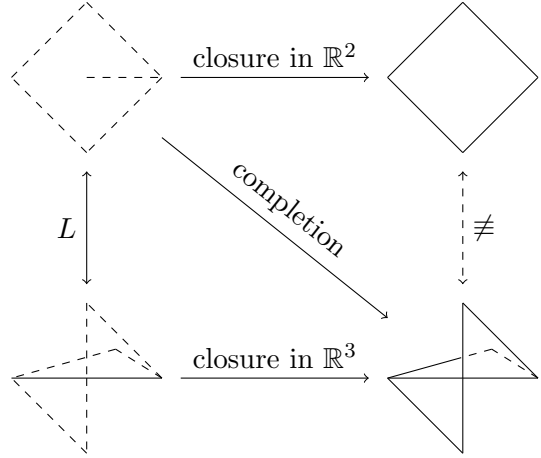
\begin{figure}[h]
        \centering
        \begin{tikzpicture}
            \draw[dashed] (-1,0) -- (0,1);
            \draw[dashed] (0,1) -- (1,0);
            \draw[dashed] (1,0) -- (0,-1);
            \draw[dashed] (0,-1) -- (-1,0);
            \draw[dashed] (0,0) -- (1,0);

            \draw[<->] (0,-5/4) -- (0,-11/4) node[midway,left] {$L$};

            \draw[->] (1,-4/5) -- (4,-16/5) node[sloped,midway, above]{completion};

            \draw[dashed,<->] (5,-5/4) -- (5,-11/4) node[midway,right] {$\not\equiv$};
            
            \draw[->] (5/4,0) -- (15/4,0) node[midway,above] {closure in $\mathbb{R}^{2}$};
            
            \draw (4,0) -- (5,1);
            \draw (5,1) -- (6,0);
            \draw (6,0) -- (5,-1);
            \draw (5,-1) -- (4,0);

            \draw[dashed] (0,-4,-1) -- (-1,-4,0);
            \draw[dashed] (1,-4,0) -- (0,-3,0);
            \draw (-1,-4,0) -- (1,-4,0);
            \draw[dashed] (0,-3,0) -- (0,-4,0);
            \draw[dashed] (0,-4,0) -- (0,-5,0);
            \draw[dashed] (0,-5,0) -- (-1,-4,0);
            \draw[dashed] (1,-4,0) -- (0,-4,-1);

            \draw[->] (5/4,-4) -- (15/4,-4) node[midway,above] {closure in $\mathbb{R}^{3}$};


            \draw[dashed] (5,-4,-1) -- (14/3,-4,-2/3);
            \draw (14/3,-4,-2/3) -- (4,-4,0);
            \draw (6,-4,0) -- (5,-3,0);
            \draw (4,-4,0) -- (6,-4,0);
            \draw (5,-3,0) -- (5,-4,0);
            \draw (5,-4,0) -- (5,-5,0);
            \draw (5,-5,0) -- (4,-4,0);
            \draw[dashed] (6,-4,0) -- (5,-4,-1);
            
        \end{tikzpicture}
        \caption{$\Omega$ and its closure in $\mathbb{R}^{2}$ vs $M$ and its completion embedded in $\mathbb{R}^{3}$}
        \label{fig:embedding}
    \end{figure}
    
    In other words, $M_{\Omega}$ is obtained by embedding $\Omega$ in the $XY$ plane of $\mathbb{R}^{3}$ and then folding the flaps given by the 1st and 4th quadrants upwards and downwards, respectively, as shown in Figure \ref{fig:embedding}. It is not difficult to see that this process (which we denote by a corresponding isometric embedding $L$) gives an isometry between $\Omega$ and $M_{\Omega}$. Moreover, the completion of $M_{\Omega}$ is simply the closure of this set in $\mathbb{R}^{3}$.

    From this, the extension of Lipschitz functions defined over $(\Omega,d)$ and the extension of the distance itself becomes evident thanks to this embedding.

\end{example} 

Now we refer to \cite{CKK17}, where a similar result is obtained in a different context.

\begin{theorem*}\cite[\textbf{Theorem 1.1}]{CKK17}
    Let $\Omega\subset E$ be a non-empty convex open set. Then, the Lipschitz free space $\lipfree{\Omega}$ is canonically isometric to the quotient space
    \[ L^{1}(\Omega;E)/X. \]
    Moreover, if $x_{0}\in\Omega$ is the base point and $x\in\Omega$ is arbitrary, then in this identification we have
    \[ \delta(x) \leftrightarrow [h] \iff h\in L^{1}(\Omega;E) \text{ and } \textup{div}(h) = \delta_{x_{0}}-\delta_{x} \text{ in } \mathcal{D}'(E), \]
    where $\delta_{y}$ denotes the Dirac distribution supported at $y$.
\end{theorem*}

It is evident the relation between this result and Theorem \ref{main2} just in the way it is written. It is important to notice that in \cite[Theorem 1.1]{CKK17}, the metric used in $\Omega$ is the one induced directly by the norm of $E$. Since $\Omega$ is convex, this metric coincides with the intrinsic metric induced by the norm, which shows that Theorem \ref{main2} is a generalization of \cite[Theorem 1.1]{CKK17}.

In both cases, Lipschitz regularity of the boundary of the domain was used at some point in the proofs. Our approach avoids completely that assumption. It can be seen, by comparing the proofs present in this paper against those in \cite{BCJ05} and \cite{CKK17}, that the use of adequate approximations is the key difference, since these methods allow us to avoid the explicit study of the behavior of the functions on the completion of $\Omega$ and restrict ourselves to the given domain.

As mentioned in the introduction, domains in finite-dimensional spaces endowed with its intrinsic metric are somewhat opposite to the case of purely $1$-unrectifiable spaces. In that sense, there are some intermediate cases for subsets of a finite-dimensional space for which we lack a characterization. Such cases are, e.g., non-connected open subsets of $E$ and closed subsets of $E$ such that the closure of its interior is strictly contained in the set itself.

Another question arises for $n$-dimensional Lipschitz manifolds. In this case, as they are locally Lipeomorphic to an open connected subset of $\mathbb{R}^{n}$, it might be possible to adapt the techniques used in the present work to be applied to the more general setting of Lipschitz manifolds.


\section*{Acknowledgements}

The author would like to thank Anton Svensson for the fruitful discussions on the topics present in this article. The author would also like to thank the referee for a careful reading and helpful remarks that improved the present work.

\section*{Funding}

The author was partially supported by Universidad de O'Higgins (112011002-PD-17706171-9)


\end{document}